\documentclass[10pt]{article}
\usepackage{latexsym}
\usepackage{amssymb}
\usepackage{amsthm}
\usepackage{amsmath}
\usepackage{mathrsfs}
\usepackage{cite}
\usepackage{textcomp}
\newtheorem{theorem}{Theorem}[section]
\newtheorem{lemma}[theorem]{Lemma}
\newtheorem{corollary}[theorem]{Corollary}
\newtheorem{remark}[theorem]{Remark}

\numberwithin{equation}{section}

\newcommand{\nb}{\nonumber}
\newcommand{\R}{{\mathscr R}}
\newcommand{\N}{{\mathscr N}}

\begin{document}

\begin{center}
{\LARGE  \bf Twists of two or multiple idempotent matrices}
\end{center}

\medskip

\begin{center}
{\Large  Yongge Tian}
\end{center}

\medskip
\begin{center}
{\footnotesize \em
Shanghai Business School, Shanghai, China \& Central University of Finance and Economics, Beijing, China}
\end{center}

\renewcommand{\thefootnote}{\fnsymbol{footnote}}
\footnotetext{{\it E-mail:} yongge.tian@gmail.com}

\noindent%
{\small {\bf Abstract.} \ In this article, we revisit some block matrix construction methods and use them to derive various general expansion formulas for calculating the ranks of matrix expressions. As applications, we derive a variety of interesting rank equalities for matrix expressions composed by idempotent matrices, and present their applications in the characterization of some matrix equalities for generalized inverses of partitioned matrices. \\

\noindent {\bf Aathematics Subject Classifications (2000):}  15A03; 15A09; 15A27; 47A05

\medskip

}

\section[1]{Introduction}

Throughout this article, let ${\mathbb C}^{m\times n}$ denote the set of all $m\times n$ complex matrices; $A^{\ast}$,  $r(A)$, and ${\mathscr R}(A)$ be the conjugate transpose, the rank, and the range (column space) of a matrix $A\in {\mathbb C}^{m\times n}$, respectively; $I_m$ be the identity matrix of order $m$; and $[\, A, \, B\,]$ be a row block matrix consisting of $A$ and $B$. We next introduce the definition and notation of generalized inverses of matrix. The Moore--Penrose inverse of $A \in {\mathbb C}^{m \times n}$, denoted by $A^{\dag}$,
is the unique matrix $X \in {\mathbb C}^{n \times m}$ satisfying the four Penrose equations
\begin{align}
{\rm (i)}  \  AXA = A,   \ \   {\rm (ii)}   \ XAX=X, \ \ {\rm (iii)}  \  (AX)^{\ast} = AX,  \ \  {\rm (iv)}   \  (XA)^{\ast} = XA.
 \label{11}
\end{align}
A matrix $X$ is called an $\{i,\ldots, j\}$-generalized inverse of $A$, denoted by $A^{(i,\ldots, j)}$, if it satisfies the $i$th,$\ldots,j$th equations. The collection of all $\{i,\ldots, j\}$-generalized inverses of $A$ is denoted by $\{A^{(i,\ldots, j)}\}$. The eight commonly-used generalized inverses of $A$ are $A^{\dag}$, $A^{(1,3,4)}$, $A^{(1,2,4)}$, $A^{(1,2,3)}$,  $A^{(1,4)}$,  $A^{(1,3)}$, $A^{(1,2)}$, and $A^{(1)}$. Furthermore, let $P_{A} =AA^{\dag}$, $E_{A} = I_m - AA^{\dag}$, and $F_{A} = I_n -A^{\dag}A$ stand for the three orthogonal projectors induced by $A$. Moreover, a matrix $X$ is called a $\{1\}$-inverse of $A$, denoted by $A^{-}$, if it satisfies $AXA = A$; the collection of all $A^{-}$ is denoted by $\{A^{-}\}$. The Drazin inverse of a square matrix $M$,  denoted by $X = M^{D}$, is defined to be the unique solution $X$ to the following three matrix equations $M^{t}XM = M^t$,  $XMX = X$ and $MX =  XM$,  where $t$ is the index of $M$, i.e., the smallest nonnegative integer $t$ such that $r(M^t) = r(M^{t+1}).$ When $t =1$, $X$ is also called the group inverse of $M$ and is denoted by $M^{\#}$. See e.g., \cite{BG,CM,RM} for more issues on generalized inverses of matrices.

The rank of matrix is a quite basic concept in linear algebra, which may be defined by different manners and can be calculated directly by transforming the matrix to certain row and/or column echelon forms. One of the most important applications of ranks of matrices is to describe singularity and nonsingularity of matrices, as well as the dimensions of row and column spaces of the matrices.  Thus, people would always be of interest in establishing various simple and valuable formulas for calculating the ranks of matrices under various assumptions. One of the best-known fundamental formulas for ranks of matrices is $r(A) = r(PAQ)$ provided $P$ and $Q$ are two nonsingular matrices, and people used it to derive numerous interesting and useful rank equalities for different choice of the matrices $A$, $P$, and $Q$. It is well known in linear algebra that people can establish rank formulas from block matrices and their elementary operations. In this article, we revisit this time-honored trick through some examples and show how to
establish various valuable formulas for calculating the ranks of matrices using the block matrix method. As applications, we solve many some matrix equality problems on idempotent matrices and generalized inverses under various assumptions.

\section[2]{How to establish rank formulas for specified block matrices}

It has a long history in linear algebra to establish simple and useful equalities for ranks of matrices using various tricky elementary operations of matrices. Especially, there is a major route to derive matrix rank formulas through various specific block matrix constructions. For instance, the following rank formulas in linear algebra
\begin{align}
r(\,I_m - A^2 \,)  & = r(\, I_m + A \,) +  r(\, I_m - A \,)  - m,
\label{hh21}
\\
r(\,A \pm A^2 \,)  & = r(A) +  r(\, I_m \pm A \,)  - m,
\label{hh22}
\\
r(\,A \pm A^3 \,)  & = r(A) +  r(\, I_m \pm A^2 \,)  - m,
\label{hh23}
\\
r[\, A(\, I_m \pm A \,)^2 \,] & =  r(A) +  r[\, (\, I_m \pm A \,)^2 \,] - m,
\label{hh24}
\\
 r(\, I_m - AB \,) + n  & = r(\, I_n - BA \,) + m,
\label{hh25}
\\
 r(\, A - AXBYA \,) + r(B) & = r(\, B - BYAXB \,) + r(A)
\label{hh27}
\end{align}
for any matrices $A$, $B$, $X$, and $Y$ of the appropriate sizes and the following rank formula
\begin{align}
r[(I_m - P - Q + QP) = m - r(P) - r(Q) + r(PQ)
\label{hh28}
\end{align}
for any two idempotent matrices $P$ and $Q$ of the same size are simple and well known, which can be established by calculating the ranks of the following specified two-by-two block matrices
\begin{align*}
& \begin{bmatrix} I_m &  I_m + A \\ I_m - A & 0 \end{bmatrix}, \ \ \begin{bmatrix} I_m & I_m \pm A \\ A & 0 \end{bmatrix},  \ \ \begin{bmatrix} I_m &  I_m \pm A^2 \\ A & 0 \end{bmatrix},  \ \ \begin{bmatrix} I_m &  (I_m \pm A)^2 \\ A & 0 \end{bmatrix},
\\
& \ \ \ \ \ \ \ \ \ \ \ \  \ \ \ \  \begin{bmatrix} I_m & A \\ B & I_n \end{bmatrix},  \ \ \ \begin{bmatrix} A & AXB \\ BYA  & B \end{bmatrix},  \ \ \ \begin{bmatrix} I_m &  Q \\ P & 0 \end{bmatrix},
\end{align*}
respectively; see, e.g., \cite{ASt,MS:1974,Sty,TS4}. These rank formulas can directly be used to characterize algebraic properties of the matrices in the formulas, such as, nullity, singularity, nonsingularity, etc., and of course are basic issues in many textbooks in linear algebra and matrix theory.
It seems more natural to consider consecutive subsequences of these block matrices and to extend such kind of rank formulas to general settings under  various assumptions. The principal issue dealt with here is to what extent these formulas generalize to cases with multiple matrices. In first half of this section, we derive three general rank formulas using block matrices composed by general solutions of several consistent linear matrix equations. We then present a variety of simple and interesting consequences for idempotent matrices in the formulas.

\begin{theorem} \label{T1}
Let $M \in {\mathbb C}^{m\times m}$ be given$,$ and assume that $X, \, Y \in {\mathbb C}^{m\times m}$
are solutions of the following three matrix equations
\begin{align}
MX = X, \ \ YM = Y, \ \  MY = XM.
\label{z1}
\end{align}
Then the rank of $X - Y$ can be calculated by the expansion formula
\begin{align}
r(\,X - Y\,) = r\!\begin{bmatrix} X \\ Y \end{bmatrix} + r[\,X, \, Y\,] - r(X) - r(Y).
\label{z2}
\end{align}
\end{theorem}

\begin{proof}
We first construct a block matrix from $X$ and $Y$ as follows
\begin{align}
N= \begin{bmatrix} - X & 0  & X
\\ 0 & Y  & Y  \\ X & Y & 0 \end{bmatrix}\!.
\label{z3}
\end{align}
Then it is easy to verify that
\begin{align}
P_1NQ_1 =\begin{bmatrix} I_m & 0  & 0
\\ 0 & I_m  & 0  \\ I_m & -I_m & I_m \end{bmatrix}\!N\!\begin{bmatrix} I_m & 0  & I_m
\\ 0 & I_m  & -I_m  \\ 0 & 0 & I_m \end{bmatrix} = \begin{bmatrix} - X & 0  & 0
\\ 0 & Y  & 0  \\ 0 & 0 & X-Y \end{bmatrix}\!,
\label{z4}
\end{align}
and from \eqref{z1} that
\begin{align}
P_2NQ_2 = \begin{bmatrix} I_m & 0  & M
\\ 0 & I_m  & 0  \\ 0 & 0 & I_m \end{bmatrix}\!N\!\begin{bmatrix} I_m & 0  & 0
\\ 0 & I_m  & 0  \\ 0 & -M & I_m \end{bmatrix} = \begin{bmatrix} 0 & 0  & X
\\ 0 & 0  & Y  \\ X & Y & 0 \end{bmatrix}\!.
\label{z5}
\end{align}
Since $P_1$,  $Q_1$,  $P_2$,  and $Q_2$ are nonsingular matrices, both \eqref{z4} and \eqref{z5} imply that
\begin{align}
& r(N) = r(P_1NQ_1) =  r\!\begin{bmatrix} - X & 0  & 0
\\ 0 & Y  & 0  \\ 0 & 0 & X-Y \end{bmatrix} = r(\,X - Y\,) + r(X) + r(Y),
\label{z6}
\\
& r(N) = r(P_2NQ_2) = r\!\begin{bmatrix} 0 & 0  & X
\\ 0 & 0  & Y  \\ X & Y & 0 \end{bmatrix} = r\!\begin{bmatrix} X \\ Y \end{bmatrix} + r[\,X, \, Y\,].
\label{z7}
\end{align}
Combining \eqref{z6} and \eqref{z7} leads to \eqref{z2}.
\end{proof}

\begin{theorem} \label{T2}
Let $A, \, B\in {\mathbb C}^{m\times m}$ be given$,$ and assume that $X, \, Y \in {\mathbb C}^{m\times m}$ are solutions of the following equations
\begin{align}
AX = X, \ \ YB = Y, \ \ AY = XB.
\label{z8}
\end{align}
Then the rank of $X - Y$ can be calculated by the expansion formula
\begin{align}
r(\,X - Y\,) = r\!\begin{bmatrix} X \\ Y \end{bmatrix} + r[\,X, \, Y\,] - r(X) - r(Y).
\label{z9}
\end{align}
\end{theorem}

\begin{proof}
It is easy to verify from \eqref{z8} and elementary matrix transformations that the rank of $N$ in \eqref{z3} is
\begin{align}
r(N) = r\!\begin{bmatrix} 0 & AY  & X
\\ 0 & Y  & Y  \\ X & Y & 0 \end{bmatrix} = r\!\begin{bmatrix} 0 & XB  & X
\\ 0 & YB  & Y  \\ X & Y & 0 \end{bmatrix} =
r\begin{bmatrix} 0 & 0  & X
\\ 0 & 0  & Y  \\ X & Y & 0 \end{bmatrix}  = r\!\begin{bmatrix} X \\ Y \end{bmatrix} + r[\,X, \, Y\,].
\label{z10}
\end{align}
Combining \eqref{z6} and \eqref{z10} leads to \eqref{z9}.
\end{proof}

\begin{theorem} \label{T3}
Let $A, \,B \in {\mathbb C}^{m\times m}$ be given$,$ and assume that $X \in {\mathbb C}^{m\times n}$ and
$Y \in {\mathbb C}^{m\times p}$ are solutions of the following equations
\begin{align}
AX = X, \ \ BY = Y, \ \ \R(X) \supseteq \R(AY), \ \ \R(Y) \supseteq \R(BX).
\label{z11}
\end{align}
Then the following rank equality holds
\begin{align}
r[\,AY, \,  BX\,] = r[\,X, \, Y\,] + r(AY) + r(BX) -  r(X) - r(Y).
\label{z12}
\end{align}
\end{theorem}

\begin{proof}
Construct a block matrix from $X$, $Y$, $AY$, and $BX$ as follows
\begin{align}
N = \begin{bmatrix} X & 0  & AY & 0
\\ 0 & Y  & 0  & BX \\ X & Y & 0 & 0\end{bmatrix}\!.
\label{z13}
\end{align}
Then it is easy to verify under \eqref{z11} that
\begin{align}
r(N) = r\!\begin{bmatrix} X & 0  & 0 & 0
\\ 0 & Y  & 0  & 0 \\ 0 & 0 & -AY & - BX\end{bmatrix} =  r[\,AY, \,  BX\,] + r(X) + r(Y),
\label{z14}
\end{align}
and that
\begin{align}
r(N) = r\!\begin{bmatrix} 0 &  -AY & AY & 0
\\ -BX & 0  & 0  & BX \\ X & Y & 0 & 0\end{bmatrix}
 = r\!\begin{bmatrix} 0 &  0 & AY & 0
\\ 0 & 0  & 0  & BX \\ X & Y & 0 & 0\end{bmatrix} =  r[\,X, \,  Y\,] + r(AY) + r(BX).
\label{z15}
\end{align}
Combining \eqref{z14} and \eqref{z15} leads to \eqref{z12}.
\end{proof}

The matrix equations in \eqref{z1}, \eqref{z8},  and \eqref{z11} are quite fundamental in matrix analysis, and have been widely studied in theory and applications,; see \cite{Mey}. Under the assumptions of these equations, \eqref{z2},  \eqref{z9}, and \eqref{z12} link the solutions of these matrix equations and their operations. In this situation, there is a very strong intrinsic mathematical motivation to establish concrete rank formulas from \eqref{z1}, \eqref{z8},  and \eqref{z11} for various solutions of the matrix equations.

Recall that a square matrix $A$ is said to be idempotent if $A^2 = A$. An idempotent matrix is often called an oblique projector whose null space is oblique to its range, in contrast to orthogonal projector, whose null space is orthogonal to its range. idempotents can be defined in more general algebraic structures,
and are important  tools in the investigation of the algebraic structures. As is known to all, idempotent matrices have strikingly simple and interesting properties, and one of such properties is that any idempotent matrix $A$ can be decomposed as $A = P{\rm diag}(I_k, 0)P^{-1}$, where $k$ is the rank of $A$, while any Hermitian idempotent matrix $A$ admits the decomposition $A = P{\rm diag}(I_k, 0)P^{\ast}$, where $P^{-1} = P^{\ast}$. Idempotent matrices arise naturally in the theory of generalized inverses of matrices, and are a class of fundamental objects of study in matrix analysis. For instance,
\begin{enumerate}
\item[{\rm (a)}] the pair of the two ordered products $AA^-$ and $A^-A$ are always idempotent matrices for any generalized inverse $A^-$ of $A$;

\item[{\rm (b)}] both $AA^{\dag}$ and $A^{\dag}A$ are Hermitian idempotent matrices for the Moore--Penrose inverse $A^{\dag}$ of $A$;

\item[{\rm (c)}] the matrix $X(X^{*}VX)^{\dag}X^{*}V$ is always idempotent. This matrix
often occurs in the weighted least-squares estimation problems in linear regression analysis.
\end{enumerate}
Many other types of matrix can be converted into idempotent matrices through some elementary operations. For instance,
\begin{enumerate}
\item[{\rm (d)}] if $A^2 = -A$, then $(-A)^2 = -A$, i.e., $-A$ is idempotent;

\item[{\rm (e)}] if $A^2 = I_m$, then $(\,I_m \pm A \,)/2$ are idempotent; if $A^2 = -I_m$,
  then $(I_m \pm iA)/2$ are idempotent;

\item[{\rm (f)}] the product $B(AB)^{\dag}A$, as well as $BC(ABC)^{\dag}A$, and $C(ABC)^{\dag}AB$ are idempotent.
\end{enumerate}
Any matrix $A$ satisfying a quadratic equation $A^2 + aA + bI_m=0$ can be
written as $[\, A - (a/2)I_m \,]^2 = (\,  a^2/4 - b\,)I_m$. If
$a^2/4  -b \neq 0$, then we can also construct an idempotent matrix
from this equality. Through these transformations, various results on idempotent
matrices can be extended to other types of quadratic matrices.

There is a substantial literature related to the approaches on idempotent matrices and related topics; see e.g., \cite{Af,ACM,ACS,AHT,Ara,BB,BBS,BBSz,BBo,BBT,Bal,BV,Bor,BGK,BS1,BS2,BS3,Bak,BT,BTh, CX,CDu,CD,Cer,Cve,Den1,Den2,Den3,Dok,DRR,DD,BES1,BES2,BES3,BES4,BES5,Erd,FRS,Fou,Gab,
Gre,Gro1,Gro2,Gro3,GT,HN,HO1,HO2,HP1,HP2,HS,HT,Ikr1,Ikr2,ITP,Kala,Kaw,KR1,KR2,KR3,KRS,KRa,KRSa,
Kol,KS1,KS2,Krup,KRSi,KSp,Laf,LMR,LWY,ND,Nis,Oml,OO,Pas,Paz1,Paz2,Paz3,PT,PSI,Rab0,Rab1,
Rab2,RSa,RR,RS,SO,Shc,Spi,Steg,Ste,TY,T2,T3,Tian:2011,Tian:2019b,TT1,TT2,Tos,Tre,Vet,
Wan,WW,Wu,XZZ,YTY,YL,ZX}, and one of the main contributions in this respect is establishing various analytical and valuable formulas for calculating ranks of various matrix expressions composed by idempotent matrices; see e.g., \cite{CCLY,KR2,TS1,TS2,TS3,ZLLY,Zuo1,Zuo2,ZZY} among others.


Armed with the results in Theorems \ref{T1}--\ref{T3}, we can establish many expansion formulas for calculating ranks of matrix expressions composed by idempotent matrices, and to present their consequences on the relationships among the idempotent matrices which, as we
shall see later, are the basis for the theory of ranks of idempotent matrices.

\begin{theorem} \label{T4}
Let $A$ and $B$ be two idempotent matrices of the same size$.$  Then the following rank equalities
\begin{align}
r[\,(AB)^k, \, (BA)^k\,] & = r[\,A, \, B\,] + r[(AB)^k] + r[(BA)^k]  - r(A) - r(B),
\label{z16}
\\
r[\,(AB)^kA, \, (BA)^kB\,] & = r[\,A, \, B\,] + r[(AB)^kA] + r[(BA)^kB]  - r(A) - r(B),
\label{z17}
\\
r\!\begin{bmatrix} (AB)^k \\ (BA)^k \end{bmatrix} & =  r\!\begin{bmatrix} A \\ B \end{bmatrix} +  r[(AB)^k] + r[(BA)^k] - r(A) - r(B),
\label{z18}
\\
r\!\begin{bmatrix} (AB)^kA \\ (BA)^kB \end{bmatrix} & =  r\!\begin{bmatrix} A \\ B \end{bmatrix} +  r[(AB)^kA] + r[(BA)^kB] - r(A) - r(B),
\label{z19}
\\
r[\,(AB)^k - (BA)^k\,] & = r\!\begin{bmatrix} (AB)^k \\ (BA)^k  \end{bmatrix} + r[\,(AB)^k, \, (BA)^k\,] -  r[(AB)^k] - r[(BA)^k],
\label{z20}
\\
r[\,(AB)^k - (BA)^k\,] &  = r\!\begin{bmatrix} A \\ B \end{bmatrix} + r[\,A, \, B\,] +  r[(AB)^k] + r[(BA)^k] - 2r(A) - 2r(B),
\label{z21}
\\
r[\,(AB)^kA - (BA)^kB\,] & = r\!\begin{bmatrix} (AB)^kA \\ (BA)^kB  \end{bmatrix} + r[\,(AB)^kA, \, (BA)^kB\,] -  r[(AB)^kA] - r[(BA)^kB],
\label{z22}
\\
r[\,(AB)^kA - (BA)^kB\,] &  = r\!\begin{bmatrix} A \\ B \end{bmatrix} + r[\,A, \, B\,] +  r[(AB)^kA] + r[(BA)^kB] - 2r(A) - 2r(B)
\label{z23}
\end{align}
hold for all integers $k \geq 1.$ In particular$,$ the following consequences hold$:$
\begin{enumerate}
\item[{\rm (a)}] $r[\,(AB)^k, \, (BA)^k\,] = r[(AB)^k] + r[(BA)^k]$ $\Leftrightarrow$ $r[\,A, \, B\,] = r(A) + r(B)$
 $\Leftrightarrow$ $\R[(AB)^k] \cap \R[(BA)^k] = \{0\}$ $\Leftrightarrow$ $\R(A) \cap \R(B) = \{0\}.$

\item[{\rm (b)}] $r[\,(AB)^k, \, (BA)^k\,]= r[\,A, \, B\,]$ $\Leftrightarrow$ $\R[(AB)^k] = \R(A)$ and $\R[(BA)^k] = \R(B).$

\item[{\rm (c)}] $(AB)^k = (BA)^k$ $\Leftrightarrow$ $\R[(AB)^k] = \R[(BA)^k]$ and $\R[(A^{\ast}B^{\ast})^k] = \R[(B^{\ast}A^{\ast}y)^k]$ $\Leftrightarrow$
$r[A, \, B\,] = r(A) + r(B) - r[(AB)^k]$ and $r\!\begin{bmatrix} A \\ B \end{bmatrix} = r(A) + r(B) - r[(BA)^k].$

\item[{\rm (d)}] $r[\,(AB)^kA, \, (BA)^kB\,] = r[(AB)^kA] + r[(BA)^kB]$ $\Leftrightarrow$ $r[\,A, \, B\,] = r(A) + r(B)$
 $\Leftrightarrow$ $\R[(AB)^kA] \cap \R[(BA)^kB] = \{0\}$ $\Leftrightarrow$ $\R(A) \cap \R(B) = \{0\}.$

\item[{\rm (e)}] $r[\,(AB)^kA, \, (BA)^kB\,]= r[\,A, \, B\,]$ $\Leftrightarrow$ $\R[(AB)^kA] = \R(A)$ and $\R[(BA)^kB] = \R(B).$

\item[{\rm (f)}] $(AB)^kA = (BA)^kB$ $\Leftrightarrow$ $\R[(AB)^kA] = \R[(BA)^kB]$ and $\R[(A^{\ast}B^{\ast})^kA^{\ast}] = \R[(B^{\ast}A^{\ast})^kB^{\ast}]$ $\Leftrightarrow$ $r[A, \, B\,] = r(A) + r(B) - r[(AB)^kA]$ and $r\!\begin{bmatrix} A \\ B \end{bmatrix} = r(A) + r(B) - r[(BA)^kB].$
\end{enumerate}
\end{theorem}

\begin{proof}
Let $X = (AB)^{k-1}A$ and $Y =(BA)^{k-1}B$, as well as $X = (AB)^k$ and $Y =(BA)^k$, as well as  respectively.
Then they satisfy \eqref{z11}. In such cases, \eqref{z12} becomes \eqref{z16} and  \eqref{z17}, respectively.
Eqs. \eqref{z18} and \eqref{z19} are established by taking transpose of \eqref{z16} and  \eqref{z17}, respectively.

Let $M =A$, $X = (AB)^k$, and $Y =(BA)^k$. Then they satisfy \eqref{z1}, thus \eqref{z2} becomes \eqref{z20}.

Let $X = (AB)^kA$ and $Y = (BA)^kB$. Then they satisfy \eqref{z8}, thus \eqref{z9} becomes \eqref{z22}.

Substituting \eqref{z16}--\eqref{z19} into \eqref{z20} and \eqref{z22} yields \eqref{z21} and \eqref{z23}, respectively.
 Results (a)--(f) follow directly from  \eqref{z16}, \eqref{z17}, and \eqref{z20}--\eqref{z23}.
 \end{proof}

Some of \eqref{z16}--\eqref{z23} were established in the literature; see, e.g., \cite{TS1,TS2,TS3,ZZY,ZLLY}.
By a similar approach, we can also establish a general rank formula associated with a family of idempotent
matrices of the same size.

\begin{theorem} \label{Th25}
Let $A_1, A_2, \ldots, A_k$ be a family of idempotent matrices of the same size$,$ and
denote
\begin{align}
A=[A_1, A_2, \ldots, A_k] \ \  and \ \ \widehat{A}_i=[A_1, \ldots, A_{i-1}, \, 0, \, A_{i+1}, \ldots,  A_k].
\label{z24}
\end{align}
Then they satisfy the following rank identity
\begin{align}
 r[A_1\widehat{A}_1, A_2\widehat{A}_2,\ldots, A_k\widehat{A}_k] = r(A_1\widehat{A}_1) + r(A_2\widehat{A}_2) + \cdots + r(A_k\widehat{A}_k)   + r(A) - r(A_1) - r(A_2) - \cdots - r(A_k).
\label{z25}
\end{align}
In particular$,$ the following results hold$:$
\begin{enumerate}
\item[{\rm (a)}] $r[A_1\widehat{A}_1, A_2\widehat{A}_2,\ldots, A_k\widehat{A}_k] = r(A_1\widehat{A}_1) + r(A_2\widehat{A}_2) + \cdots + r(A_k\widehat{A}_k)$ if and only if $r(A) =  r(A_1) + r(A_2) + \cdots + r(A_k).$

 \item[{\rm (b)}] $r[A_1\widehat{A}_1, A_2\widehat{A}_2,\ldots, A_k\widehat{A}_k] = r(A)$ if and only if $\R(A_i\widehat{A}_i) = \R(A_i),$ $i = 1, 2, \ldots, k.$

 \item[{\rm (c)}] If $A_1\widehat{A}_1 = A_2\widehat{A}_2 = \ldots = A_k\widehat{A}_k = 0,$ then $r(A) = r(A_1) + r(A_2) + \cdots + r(A_k).$

 \item[{\rm (d)}] $r(A) \geq  r(A_1) + r(A_2) + \cdots + r(A_k) - r(A_1\widehat{A}_1) - r(A_2\widehat{A}_2) - \cdots - r(A_k\widehat{A}_k)$ holds$.$
\end{enumerate}
\end{theorem}

\begin{proof}
From the given matrices, we construct a block matrix as follows
\begin{align}
M = \begin{bmatrix} A_1 & 0 & \cdots  & 0 & A_1\widehat{A}_1 & 0 & \cdots  & 0
\\
0 & A_2  & \cdots  & 0 & 0 & A_2\widehat{A}_2  & \cdots  & 0
\\
\vdots  & \vdots & \ddots  & \vdots & \vdots  & \vdots & \ddots  & \vdots
\\
0  & 0 & \cdots & A_k  & 0  & 0 & \cdots & A_k\widehat{A}_k
\\
 A_1 & A_2 & \cdots  & A_k & 0 & 0 & \cdots  & 0
\end{bmatrix}.
\label{z26}
\end{align}
We then apply elementary block matrix operations to this $X$ to obtain the following rank equality
\begin{align}
r(X)& = r\!\begin{bmatrix} A_1 & 0 & \cdots  & 0 &0 & 0 & \cdots  & 0
\\
0 & A_2  & \cdots  & 0 & 0 & 0 & \cdots  & 0
\\
\vdots  & \vdots & \ddots  & \vdots & \vdots  & \vdots & \ddots  & \vdots
\\
0  & 0 & \cdots & A_k  & 0  & 0 & \cdots & 0
\\
 0 & 0 & \cdots  & 0 & -A_1\widehat{A}_1 & -A_2\widehat{A}_2 & \cdots  & -A_k\widehat{A}_k
\end{bmatrix} \nb
\\
& = r(A_1) + r(A_2) + \cdots + r(A_k) + r[A_1\widehat{A}_1, A_2\widehat{A}_2,\ldots, A_k\widehat{A}_k].
\label{z27}
\end{align}
Also by elementary block matrix operations and the idempotency of $A_1, A_2, \ldots,  A_k$, we obtain the following rank equality
\begin{align}
r(M) & = r\!\begin{bmatrix} 0 & -A_1A_2 & \cdots  & -A_1A_k & A_1\widehat{A}_1 & 0 & \cdots  & 0
\\
-A_2A_1 & 0  & \cdots  & -A_2A_k & 0 & A_2\widehat{A}_2  & \cdots  & 0
\\
\vdots  & \vdots & \ddots  & \vdots & \vdots  & \vdots & \ddots  & \vdots
\\
-A_kA_1  & -A_kA_2  & \cdots & 0  & 0  & 0 & \cdots & A_k\widehat{A}_k
\\
 A_1 & A_2 & \cdots  & A_k & 0 & 0 & \cdots  & 0
\end{bmatrix} \nb
\\
& = r\!\begin{bmatrix} 0 & 0 & \cdots  & 0 & A_1\widehat{A}_1 & 0 & \cdots  & 0
\\
0 & 0  & \cdots  & 0 & 0 & A_2\widehat{A}_2  & \cdots  & 0
\\
\vdots  & \vdots & \ddots  & \vdots & \vdots  & \vdots & \ddots  & \vdots
\\
0  & 0 & \cdots & 0  & 0  & 0 & \cdots & A_k\widehat{A}_k
\\
 A_1 & A_2 & \cdots  & A_k & 0 & 0 & \cdots  & 0
\end{bmatrix} \nb
\\
& = r(A_1\widehat{A}_1) + r(A_2\widehat{A}_2) + \cdots + r(A_k\widehat{A}_k)  + r(A).
\label{z28}
\end{align}
Combining \eqref{z27} and \eqref{z28} leads to \eqref{z25}. Results (a) and (b) follow directly from  \eqref{z25}.
\end{proof}

Eq. \eqref{z25} shows that all idempotent matrices are linked one another through certain simple but nontrivial rank
formulas, so that we can conveniently use them to discuss relationships among all idempotent matrices under various
circumstances. It is easy to see that \eqref{z25} for $k =2$ and \eqref{z16} for $k =1$ are the same. For $k =3$ in \eqref{z25}, we obtain the following appealing results on the relationships among any three idempotent matrices of the same size.

\begin{corollary} \label{TW26}
Let $A,$ $B,$ and $C$ be three idempotent matrices of the same size$.$  Then
\begin{align}
r[\,A, \, B, \, C\,] & = r(A) + r(B) + r(C) - r[\,AB, \, AC\,] - r[\, BA, \, BC\,] - r[\, CA, \, CB\,] \nb
\\
&  \ \ \ + r[AB, \, AC, \, BA, \, BC, \, CA, \, CB].
\label{z29}
\end{align}
If $AB = BA,$  $AC = CA,$  and  $BC = CB,$ then
\begin{align}
r[\,A, \, B, \, C\,] = r(A) + r(B) + r(C) - r[\,AB, \, AC\,] - r[\, BA, \, BC\,] - r[\, CA, \, CB\,] + r[AB, \, AC, \, BC\,].
\label{z30}
\end{align}
In particular$,$ the following results hold$.$
\begin{enumerate}
\item[{\rm (a)}] $r[\,A, \, B, \, C\,]= r(A) + r(B) + r(C)$ if and only if $r[AB, \, AC, \, BA, \, BC, \, CA, \, CB] = r[\,AB, \, AC] + r[\, BA, \, BC\,] + r[\, CA, \, CB\,].$

 \item[{\rm (b)}] $r[AB, \, AC, \, BA, \, BC, \, CA, \, CB]  = r[\,A, \, B, \, C\,]$ if and only if
 $\R[\,AB, \, AC\,]  = \R(A),$ $\R[\, BA, \, BC\,] = \R(B),$  and $\R[\, CA, \, CB\,] = \R(C).$

\item[{\rm (c)}] If $AB = BA =AC = CA = BC = CB =0,$ then $r[\,A, \, B, \, C\,] = r(A) + r(B) + r(C).$

\item[{\rm (d)}] $r[\,A, \, B, \, C\,]  \geq r(A) + r(B) + r(C) - r[\,AB, \, AC\,] - r[\, AB, \, BC\,] - r[\, AC, \, BC\,]$ holds$.$
    \end{enumerate}
\end{corollary}

\begin{corollary} \label{TW27}
Let $A,$ $B,$ and $C$ be three matrices with the same row number$,$ and denote $P_{A} = AA^{\dag},$
$P_{B} = BB^{\dag},$ and $P_{C} = CC^{\dag}.$ Then
\begin{align}
r[\,A, \, B\,] & = r(A) + r(B) - r(P_{A}P_{B}) - r(P_{B}P_{A}) + r[P_{A}P_{B}, \, P_{B}P_{A} \,],
\label{}
\\
r[\,A, \, B, \, C\,] & = r(A) + r(B) + r(C) - r[\,P_{A}P_{B}, \, P_{A}P_{C}\,] - r[\, P_{B}P_{A}, \, P_{B}P_{C}\,] - r[\, P_{C}P_{A}, \, P_{C}P_{B}\,] \nb
\\
&  \ \ \ + r[P_{A}P_{B}, \, P_{A}P_{C}, \, P_{B}P_{A}, \, P_{B}P_{C}, \,
P_{C}P_{A}, \, P_{C}P_{B}].
\label{}
\end{align}
In particular$,$ the following results hold$.$
\begin{enumerate}
\item[{\rm (a)}] $r[\,A, \, B\,]  = r(A) + r(B)$ $\Leftrightarrow$ $r[P_AP_B, \, P_BP_A \,] = r(P_AP_B) + r(P_BP_A)$  $\Leftrightarrow$ ${\mathscr R}(A) \cap {\mathscr R}(B) =\{0\}$
     $\Leftrightarrow$ ${\mathscr R}(P_{A}P_{B}) \cap {\mathscr R}(P_{B}P_{A}) =\{0\}.$

\item[{\rm (b)}] $r[\,A, \, B\,]  = r(A) + r(B) - r(P_{A}P_{B})$ $\Leftrightarrow$ $r[P_AP_B, \, P_BP_A \,] = r(P_AP_B) = r(P_BP_A)$ $\Leftrightarrow$ $\R(P_AP_B) = \R(P_BP_A)$ $\Leftrightarrow$ $P_AP_B = P_BP_A.$

\item[{\rm (c)}] $r[\,A, \, B\,] = r[P_{A}P_{B}, \, P_{B}P_{A} \,]$ $\Leftrightarrow$
$r(A^*B) = r(A) = r(B).$

\item[{\rm (d)}] $r[\,A, \, B, \, C\,] = r(A) + r(B) + r(C)$ $\Leftrightarrow$
$r[P_{A}P_{B}, \, P_{A}P_{C}, \, P_{B}P_{A}, \, P_{B}P_{C}, \,
P_{C}P_{A}, \, P_{C}P_{B}] = r[\,P_{A}P_{B}, \, P_{A}P_{C}\,] + r[\, P_{B}P_{A}, \, P_{B}P_{C}\,] + r[\, P_{C}P_{A}, \, P_{C}P_{B}\,].$

\item[{\rm (e)}] $r[\,A, \, B, \, C\,] = r(A) + r(B) + r(C) - r(P_{A}P_{B}) - r(P_{A}P_{C}) - r(P_{B}P_{C})$ $\Leftrightarrow$ $r[P_{A}P_{B}, \, P_{A}P_{C}, \, P_{B}P_{A}, \, P_{B}P_{C}, \,
P_{C}P_{A}, \, P_{C}P_{B}] = r[\,P_{A}P_{B}, \, P_{A}P_{C}\,] + r[\, P_{B}P_{A}, \, P_{B}P_{C}\,] + r[\, P_{C}P_{A}, \, P_{C}P_{B}\,] - r(P_{A}P_{B}) - r(P_{A}P_{C}) - r(P_{B}P_{C})$.
\end{enumerate}
\end{corollary}

One of the well-known problems in the theory of generalized inverses is to determine the relationships
among generalized inverses of a block and its submatrices. For instance, $\begin{bmatrix} A^{-} \\ B^{-}
\end{bmatrix}$ is a generalized inverse of $[\,A, \, B\,]$ if and only if
$[\,A, \, B\,]\begin{bmatrix} A^{-} \\ B^{-} \end{bmatrix}[\,A, \, B\,] =
[\,A, \, B\,]$ by definition. On the other hand, it is easy to verify that
\begin{align}
[\,A, \, B\,] - [\,A, \, B\,]\begin{bmatrix} A^{-} \\ B^{-}
\end{bmatrix}[\,A, \, B\,] & = [\,A, \, B\,] - [\,(AA^{-} +  BB^{-})A, \,(AA^{-} +  BB^{-})B\,] \nb
\\
& = - [\,BB^{-}A,\, AA^{-}AB\,]
\label{z31a}
\end{align}
holds for all $A^{-}$ and $B^{-}.$ Applying \eqref{z25} to \eqref{z31a} yields the following result.

\begin{corollary} \label{TW28}
Let $A \in {\mathbb C}^{m\times n}$  and $B \in {\mathbb C}^{m\times p}.$ Then the two matrices and their generalized inverses
$ A^{-}$ and $B^{-}$ satisfy the following rank identity
\begin{align}
r\!\left([\,A, \, B\,] - [\,A, \, B\,]\begin{bmatrix} A^{-} \\ B^{-}
\end{bmatrix}[\,A, \, B\,] \right) =  r(AA^{-}B) + r(BB^{-}A) + r[\,A, \, B\,] - r(A) - r(B).
\label{z32}
\end{align}
In particular$,$ the following results hold$.$
\begin{enumerate}
\item[{\rm (a)}] The maximum and minimum ranks of \eqref{z31a} with respect to $A^{-}$ and $B^{-}$ are given by
\begin{align}
& \max_{A^{-}, B^{-}} r\!\left([\,A, \, B\,] - [\,A, \, B\,]\begin{bmatrix} A^{-} \\ B^{-}
\end{bmatrix}[\,A, \, B\,] \right) = r[A, \, B\,] - |r(A) - r(B)|,
\label{z33}
\\
& \min_{A^{-}, B^{-}} r\!\left([\,A, \, B\,] - [\,A, \, B\,]\begin{bmatrix} A^{-} \\ B^{-}
\end{bmatrix}[\,A, \, B\,] \right)  =  r(A) + r(B) - r[\,A,\, B\,] = \dim[\R(A) \cap \R(B)].
\label{z34}
\end{align}

\item[{\rm (b)}] $\{[\,A, \, B\,]^{-}\} \cap  \left\{\begin{bmatrix} A^{-} \\ B^{-}
\end{bmatrix} \right\} \neq \emptyset \Leftrightarrow r[\,A,\, B\,] = r(A) + r(B) \Leftrightarrow \R(A) \cap \R(B) = \{0\}.$

\item[{\rm (c)}] $\{[\,A, \, B\,]^{-}\} \supseteq  \left\{\begin{bmatrix} A^{-} \\ B^{-}
\end{bmatrix} \right\} \Leftrightarrow r[A, \, B\,] =|r(A) - r(B)|$ $\Leftrightarrow$ $A = 0$ or $ B= 0.$
\end{enumerate}
\end{corollary}

\begin{proof}
Noting that both $AA^{-}$ and $BB^{-}$ are idempotent and applying
\eqref{z16} to \eqref{z31a}, we obtain
\begin{align}
r[\,AA^{-}BB^{-},\, BB^{-}AA^{-}\,] & =   r(AA^{-}BB^{-}) + r(BB^{-}AA^{-}) + r[\,AA^{-},\, BB^{-}\,] -
r(AA^{-}) - r(BB^{-}) \nb
\\
& =   r(AA^{-}B) + r(BB^{-}A) + r[\,A, \, B\,] - r(A) - r(B).
\label{z36}
\end{align}
Applying the two known rank formulas
\begin{align}
\max_{A^{-}} r(\, D - CA^{-}B \,) &= \min  \left\{ r[ \, C, \, D \, ], \ \
r\!\begin{bmatrix}  B \\ D  \end{bmatrix}\!,  \ \
 r\!\begin{bmatrix}  A  & B \\ C   & D
\end{bmatrix} -r(A)
 \right\}\!,
\label{z37}
\\
\min_{A^{-}}r(\,D - CA^{-}B \,)&= r(A) + r[\, C, \, D \,] +
r\!\begin{bmatrix} B  \\ D \end{bmatrix} + r\begin{bmatrix} A  &  B  \\ C & D \end{bmatrix}  -
 r\!\begin{bmatrix} A  & 0  & B \\ 0   & C  & D
 \end{bmatrix} -
  r\!\begin{bmatrix}  A  &  0  \\ 0
 & B  \\ C & D  \end{bmatrix}
\label{z38}
\end{align}
in \cite{T1} to $AA^{-}B$ and $BB^{-}A$ gives
\begin{align}
& \max_{A^{-}} r(AA^{-}B) = \max_{B^{-}} r(BB^{-}A) = \min \{r(A), \ \  r(B)\},
\label{z39}
\\
& \min_{A^{-}} r(AA^{-}B) = \min_{B^{-}} r(BB^{-}A) = r(A) + r(B) - r[\,A, \, B\,].
\label{z40}
\end{align}
Substituting \eqref{z39} and \eqref{z40} into \eqref{z36} yields
\begin{align}
 \max_{A^{-}, B^{-}} r[\,AA^{-}BB^{-},\, BB^{-}AA^{-}\,] & =  2\min \{r(A), \ \  r(B)\} + r[\,A, \, B\,] - r(A) - r(B) \nb
 \\
 & =  r[A, \, B\,] - |r(A) - r(B)|,
\label{z41}
\\
\min_{A^{-}, B^{-}} r[\,AA^{-}BB^{-},\, BB^{-}AA^{-}\,] & = 2r(A) + 2r(B) - 2r[\,A, \, B\,] + r[\,A, \, B\,] - r(A) - r(B)  \nb
\\
& = r(A) + r(B) - r[\,A, \, B\,].
\label{z42}
\end{align}
Combining \eqref{z36} with \eqref{z41} and \eqref{z42} leads to \eqref{z33} and \eqref{z34}, respectively.
\end{proof}

More rank formulas can be derived from \eqref{z25} and their variations.  For example,
for any $A \in {\mathbb C}^{m\times n}$, $B \in {\mathbb C}^{m\times p}$, and $C \in {\mathbb C}^{m\times q},$ the following rank equality
\begin{align}
& r[\,AA^{-}[\,B, \, C\,], \, BB^{-}[\,A, \, C\,], \, CC^{-}[\,A, \,B\,]\,] \nb
\\
& = r[\,A,\, B, \,C\,] +  r[\,AA^{-}B, \, AA^{-}C\,] +r[\,BB^{-}A, \, BB^{-}C\,] + r[\,CC^{-}A, \, CC^{-}B\,] - r(A) - r(B) - r(C)
\label{z43}
\end{align}
holds for all $A^{-}$, $B^{-}$, and $C^{-}$. Thus it is easy to easy to derive from  \eqref{z38} and \eqref{z43} that
\begin{align}
 \min_{A^1, B^{-}, C^{-}} r[\,AA^{-}[\,B, \, C\,], \, BB^{-}[\,A, \, C\,], \, CC^{-}[\,A, \,B\,]\,] = r[\,A,\, B\,] + r[\,A,\,C\,] + r[\,B, \,C\,] -  2r[A,\, B, \,C].
\label{z44}
\end{align}
On the other hand, it was shown in \cite{Tian:2019} that
\begin{align}
 \dim(\R[\,A,\, B\,]\cap \R[\,A,\, C\,] \cap \R[\,B, \,C\,])  =
r[\,A,\, B\,] + r[\,A,\,C\,] + r[\,B, \,C\,] - 2r[\,A,\, B, \,C\,].
 \label{z45}
\end{align}
Thus we also have the following matrix rank minimization equality
\begin{align}
 \min_{A^{-}, B^{-}, C^{-}} r[\,AA^{-}[\,B, \, C\,], \, BB^{-}[\,A, \, C\,], \, CC^{-}[\,A, \, B\,]\,] = \dim(\R[\,A, \, B\,]\cap \R[\,A, \, C\,] \cap \R[\,B, \, C\,]).
\label{z46}
\end{align}
Prompted by \eqref{z31a}, we obtain the following equality
\begin{align}
[\,A, \, B, \, C\,] - [\,A, \, B,\, C\,]\begin{bmatrix} A^{-} \\ B^{-} \\ C^{-}
\end{bmatrix}[\,A, \, B, \, C\,] = [\,(\,BB^{-} +  CC^{-} \,)A, \, (\,AA^{-} + CC^{-} \,)B, \,  (\,AA^{-} + BB^{-} \,)C\,],
\label{z47}
\end{align}
where $AA^{-}$, $BB^{-}$, and $CC^{-}$ are idempotent matrices. In this case, it would be of interest to establish expansion formulas for calculating the rank of the right-hand side of \eqref{z47}.

\section[3]{Matrix identities composed by two or three idempotent matrices and their applications}

It has been noticed that two or more given idempotent matrices may satisfy various identities, while
these identities can be used to characterize algebraic properties of matrix expressions composed by idempotent matrices. In this section, we first revisit two known identities composed by two idempotent matrices, and then to establish a variety of novel identities composed by two or three idempotent matrices and their applications, including inverses, Moore--Penrose generalized inverses, Drazin generalized inverses, rank, range, and null spaces of
the matrix expressions.

\begin{theorem} \label{TK31}
Let $A$ and $B$ be two idempotent matrices of the order $m.$  Then the following two matrix identities
\begin{align}
&  (A - B)^2 +  (A + B - I_m)^2  = I_m,
\label{v31}
\\
& AB + BA + 4^{-1}I_m  = (\, A + B - 2^{-1}I_m  \,)^2,
\label{v32}
\end{align}
and the following five rank formulas
\begin{align}
&  r[(A - B)^2] =  r(A + B) +  r(2I_m - A - B) - m,
\label{v33}
\\
&  r[(I_m - A - B )^2] =  r(I_m + A - B) +  r(I_m - A + B) - m,
\label{v34}
\\
& r(AB + BA) = r(I_m - A - B ) + r(A + B) - m,
\label{v35}
\\
& r(I_m - AB - BA) = r[(\sqrt{5} -1)/2I_m + A + B] + r[(\sqrt{5} + 1)/2I_m - A - B] - m,
\label{v36}
\\
& r(2I_m - AB - BA) = r(I_m + A + B ) + r(2I_m - A - B) - m
\label{v37}
\end{align}
hold$.$ In particular$,$ the following facts hold
\begin{align}
&  (A - B)^2  =0 \Leftrightarrow  (I_m - A - B )^2  = I_m  \Leftrightarrow r(A + B) +  r(2I_m - A - B) = m,
\label{v38}
\\
&  (A - B)^2  = 2^{-1} I_m \Leftrightarrow  (I_m - A - B )^2  =  2^{-1}I_m,
\label{v39}
\\
&  (A - B)^2  = I_m \Leftrightarrow  (I_m - A - B )^2  = 0 \Leftrightarrow r(I_m + A - B) +  r(I_m - A + B) - m,
\label{v310}
\\
& AB + BA =  - 2I_m    \Leftrightarrow  (\, A + B - 2^{-1}I_m  \,)^2 = - \frac{7}{4}I_m,
\label{v311}
\\
& AB + BA =  -I_m    \Leftrightarrow  (\, A + B - 2^{-1}I_m  \,)^2 = - \frac{3}{4}I_m,
\label{v312}
\end{align}
\begin{align}
& AB + BA =  -4^{-1}I_m  \Leftrightarrow (\, A + B - 2^{-1}I_m  \,)^2 =0,
\label{v313}
\\
& AB + BA = 0  \Leftrightarrow  (\, A + B - 2^{-1}I_m  \,)^2 = 4^{-1}I_m \Leftrightarrow  r(I_m - A - B ) + r(A + B) = m,
\label{v314}
\\
& AB + BA = \frac{3}{4}I_m   \Leftrightarrow  (\, A + B - 2^{-1}I_m  \,)^2 = I_m,
\label{v315}
\\
& AB + BA = I_m   \Leftrightarrow  (\, A + B - 2^{-1}I_m  \,)^2 = \frac{5}{4}I_m   \Leftrightarrow r[(\sqrt{5} -1)/2I_m + A + B] + r[(\sqrt{5} + 1)/2I_m - A - B] = m,
\label{v316}
\\
& AB + BA = 2I_m   \Leftrightarrow (\, A + B - 2^{-1}I_m  \,)^2 = \frac{9}{4}I_m  \Leftrightarrow  r(I_m + A + B ) + r(2I_m - A - B) = m,
\label{v317}
\end{align}
and
\begin{align}
& r(A - B) =m  \Leftrightarrow  r(A + B) =  r(2I_m - A - B) = m,
\label{vv318}
\\
& r(I_m - A - B )  = m  \Leftrightarrow r(I_m + A - B) =  r(I_m - A + B) = m,
\label{vv319}
\\
& r(AB + BA) =m \Leftrightarrow r(I_m - A - B ) = r(A + B) = m,
\label{vv320}
\\
& r(I_m - AB - BA) = m \Leftrightarrow  r[(\sqrt{5} -1)/2I_m + A + B] = r[(\sqrt{5} + 1)/2I_m - A - B] = m,
\label{vv321}
\\
& r(2I_m - AB - BA) = m \Leftrightarrow r(I_m + A + B ) = r(2I_m - A - B) = m.
\label{vv322}
\end{align}
\end{theorem}

\begin{proof}
Eqs.\,\eqref{v31} and \eqref{v32} follow from direct calculations, where \eqref{v31} was first given in \cite{Kat}; see also \cite{ASS,Bor,Nis,Sim}. Applying \eqref{hh21} to \eqref{v31} and \eqref{v32} yields
\eqref{v33}--\eqref{v37}. Eqs. \eqref{v38}--\eqref{vv322} are direct consequences of \eqref{v31}--\eqref{v37}.
\end{proof}

We next give a group of matrix identities composed by some linear combinations of two idempotent matrices and their products.

\begin{theorem}\label{TK32}
Let $A$ and $B$ be two idempotent matrices of the order $m,$ and let $k$
  be any positive integer$.$ Then the following factorization equalities
\begin{align}
\alpha AB + \beta BA & = (\, \alpha A +  \beta B \,)(\, A + B - I_m \,)
= (\, A + B - I_m \,) (\, \beta A +  \alpha B \,),
\label{ff31}
\\
\alpha ABA + \beta BAB & = (\, \alpha A +  \beta B \,)(\, A + B - I_m \,)^2
= (\, A + B - I_m \,)^2 (\, \beta A +  \alpha B \,),
\label{ff32}
\\
\alpha (AB)^k + \beta (BA)^k & = (\, \alpha A +  \beta B \,)(\, A + B - I_m \,)^{2k-1}
= (\, A + B - I_m \,)^{2k-1}(\, \beta A +  \alpha B \,),
\label{ff33}
\\
\alpha (ABA)^k + \beta (BAB)^k & = (\, \alpha A +  \beta B \,)(\, A + B - I_m \,)^{2k}
= (\, A + B - I_m \,)^{2k}(\, \beta A +  \alpha B \,)
\label{ff34}
\end{align}
hold for any two scalars $\alpha$ and $\beta.$  In particular$,$ $\alpha AB + \beta BA$ is nonsingular $\Leftrightarrow$  $\alpha (AB)^k + \beta (BA)^k$ is nonsingular $\alpha ABA + \beta BAB$ is nonsingular $\Leftrightarrow$ $\alpha (AB)^k + \beta (BA)^k$ is nonsingular
$\Leftrightarrow$ $\alpha (ABA)^k + \beta (BAB)^k$ is nonsingular $\Leftrightarrow$ both
$\alpha A + \beta B$ and $A + B - I_m$ are  nonsingular; in which cases$,$  the following equalities hold
\begin{align}
(\alpha AB + \beta BA)^{-1} & = (\, A + B - I_m \,)^{-1}(\, \alpha A +  \beta B \,)^{-1}
= (\, \beta A +  \alpha B \,)^{-1}(\, A + B - I_m \,)^{-1},
\label{dd37}
\\
(\alpha ABA + \beta BAB)^{-1} & = (\, A + B - I_m \,)^{-2}(\, \alpha A +  \beta B \,)^{-1}
=  (\, \beta A +  \alpha B \,)^{-1}(\, A + B - I_m \,)^{-2},
\label{dd38}
\\
[\alpha (AB)^k + \beta (BA)^k]^{-1} & = (\, A + B - I_m \,)^{-2k+1}(\, \alpha A +  \beta B \,)^{-1}
= (\, \beta A +  \alpha B \,)^{-1}(\, A + B - I_m \,)^{-2k+1},
\label{dd39}
\\
[\alpha (ABA)^k + \beta (BAB)^k]^{-1} & = (\, A + B - I_m \,)^{-2k}(\, \alpha A +  \beta B \,)^{-1}
= (\, \beta A +  \alpha B \,)^{-1}(\, A + B - I_m \,)^{-2k}.
\label{dd310}
\end{align}
\end{theorem}

\begin{proof}
Eqs.\,\eqref{ff31}--\eqref{ff34} follow from direct expansions and simplifications. Eqs.\,\eqref{dd37}--\eqref{dd310} follow from
\eqref{ff31}--\eqref{ff34}.
\end{proof}

It is no doubt that \eqref{ff31}--\eqref{ff34} can be used to approach performances of the matrix expressions on the left-hands under
various assumptions, such as, ranks, ranges, nullity, r-potency, nilpotency, nonsingularity, inverses, generalized inverses, norms, etc.
We next give some special cases of \eqref{ff31}--\eqref{ff34} and their variations, and present interesting consequences.

\begin{corollary}\label{TK33}
Let $A$ and $B$ be two idempotent matrices of the order $m,$ and let $k$
  be a positive integer$.$ Then$,$ the following matrix identities hold
\begin{align}
AB - BA & =  (\, A - B \,)(\, A + B - I_m \,) = - (\, A + B - I_m
\,)(\, A - B \,),
\label{w3}
\\
AB + BA & = (\, A + B \,)(\, A + B - I_m  \,) =  (\, A + B -
I_m\,)(\, A + B\,),
 \label{w4}
\\
ABA - BAB & = (\, A - B \,)(\, A + B - I_m \,)^2 = (\, A + B - I_m \,)^2(\, A - B \,),
\label{w14}
\\
ABA + BAB & = (\, A + B \,)(\, A + B - I_m \,)^2 =(\, A + B - I_m \,)^2(\, A + B \,).
\label{w15}
\end{align}
\begin{align}
(\,AB - BA\,)^k  & = (-1)^{k(k-1)/2}(\, A - B \,)^k(\,  A + B - I_m \,)^k  = (-1)^{k(k-1)/2}(\, I_m - A - B \,)^k(\, A  - B \,)^k,
\label{w6}
\\
(\,AB + BA \,)^k & = (\, A + B  \,)^k(\,  A + B - I_m \,)^k = (\,  A + B - I_m \,)^k(\, A + B \,)^k, \ \ \ \ \ \
\label{w7}
\\
(\, ABA - BAB\,)^k & = (\,  A - B \,)^k(\, A + B - I_m \,)^{2k} = (\,  A + B - I_m\,)^{2k}(\,  A - B\,)^k,
\label{w8}
\\
(\,ABA + BAB\,)^k & = (\, A + B  \,)^k(\,  A + B - I_m  \,)^{2k} = (\,  A + B - I_m \,)^{2k}(\,  A  +  B\,)^k,
\label{w9}
\\
(AB)^k - (BA)^k & = (\, A - B \,)(\,  A + B - I_m  \,)^{2k-1} = - (\, A + B - I_m \,)^{2k-1}(\, A  - B\,),
\label{w10}
\\
(AB)^k + (BA)^k & = (\, A + B \,)(\,  A + B - I_m  \,)^{2k-1} = (\, A + B - I_m \,)^{2k-1}(\, A + B \,),
\label{w11}
\\
  (ABA)^k - (BAB)^k & =  (\, A - B \,)(\,  A + B - I_m  \,)^{2k} = (\, A + B - I_m\,)^{2k}(\, A - B\,),
\label{w12}
\\
(ABA)^k + (BAB)^k & = (\, A + B \,)(\, A + B - I_m  \,)^{2k}  =(\,  A + B - I_m \,)^{2k}(\, A + B\,).
\label{w13}
\end{align}
In addition$,$ the following matrix identities hold
\begin{align}
& AB - BA  + (AB)^2 - (BA)^2 + \cdots + (AB)^k - (BA)^k  \nb
\\
& = (A - B)[(\, A + B - I_m \,) + (\,  A + B - I_m  \,)^3  + \cdots +  (\,  A + B - I_m  \,)^{2k-1}] \nb
\\
& = [(\, A + B - I_m \,) + (\,  A + B - I_m  \,)^3  + \cdots +  (\,  A + B - I_m  \,)^{2k-1}](B - A),
\\
& AB + BA  + (AB)^2 + (BA)^2 + \cdots + (AB)^k + (BA)^k  \nb
\\
& = (A + B)[(\, A + B - I_m \,) + (\,  A + B - I_m  \,)^3  + \cdots +  (\,  A + B - I_m  \,)^{2k-1}] \nb
\\
& = [(\, A + B - I_m \,) + (\,  A + B - I_m  \,)^3  + \cdots +  (\,  A + B - I_m  \,)^{2k-1}](A+B),
\\
& ABA - BAB  + (ABA)^2 - (BAB)^2 + \cdots + (ABA)^k - (BAB)^k  \nb
\\
& = (A - B)[(\, A + B - I_m \,)^2 + (\,  A + B - I_m  \,)^4  + \cdots +  (\,  A + B - I_m  \,)^{2k}] \nb
\\
& = [(\, A + B - I_m \,)^2 + (\,  A + B - I_m  \,)^4  + \cdots +  (\,  A + B - I_m  \,)^{2k}](A - B),
\\
& ABA + BAB  + (ABA)^2 + (BAB)^2 + \cdots + (ABA)^k + (BAB)^k  \nb
\\
& = (A + B)[(\, A + B - I_m \,)^2 + (\,  A + B - I_m  \,)^4  + \cdots +  (\,  A + B - I_m  \,)^{2k}] \nb
\\
& = [(\, A + B - I_m \,)^2 + (\,  A + B - I_m  \,)^4  + \cdots +  (\,  A + B - I_m  \,)^{2k}](A + B).
\end{align}
\end{corollary}

\begin{corollary}\label{TK34}
Let $A$ and $B$ be two idempotent matrices of the order $m,$ and
let $k$ be a positive integer$.$ Then$,$ the following results hold$.$
\begin{enumerate}
\item[{\rm (a)}] ${\mathscr R}[\, (AB - BA)^k\,] \subseteq
{\mathscr R}[\,(A - B)^k\,]$ and ${\mathscr R}[\,(AB - BA)^k\,]
\subseteq {\mathscr R}[\,(A + B - I_m)^k\,];$

\item[{\rm (b)}] ${\mathscr R}[\, (AB + BA)^k\,] \subseteq
{\mathscr R}[\,(A + B)^k\,]$ and ${\mathscr R}[\,(AB + BA)^k\,]
\subseteq {\mathscr R}[\,(A + B - I_m)^k\,];$

\item[{\rm (c)}]  ${\mathscr N}[\,(A - B)^k\,] \subseteq {\mathscr
N}[\, (AB - BA)^k\,]$ and ${\mathscr N}[\,(A + B -I_m)^k\,]
\subseteq {\mathscr N}[\, (AB - BA)^k\,];$

\item[{\rm (d)}]  ${\mathscr N}[\,(A + B)^k\,] \subseteq {\mathscr
N}[\, (AB + BA)^k\,]$ and ${\mathscr N}[\,(A + B -I_m)^k\,]
\subseteq {\mathscr N}[\, (AB + BA)^k\,].$
\end{enumerate}
\end{corollary}

\begin{corollary}\label{TK35}
  Let $A$ and $B$ be two idempotent matrices of the order $m,$ and let
$k$ be a positive integer$.$ Then$,$ the following matrix identities hold
\begin{align}
& (\, AB - BA \,)^D   = (\, A - B \,)^D(\, A + B - I_m \,)^D  =  - (\, A + B - I_m \,)^D(\, A - B \,)^D,
\label{w24}
\\
& (\, AB + BA \,)^D  = (\, A + B \,)^D(\, A + B - I_m  \,)^D  = (\, A + B - I_m \,)^D(\,  A + B\,)^D,
\label{w25}
\\
& \left[\, (AB)^k - (BA)^k  \,\right]^D  =
(\, A - B \,)^D \left[\, (\,  A + B - I_m  \, )^D \,\right]^{2k-1}  = - \left[(\, A + B - I_m \,)^D \right]^{2k-1}(\, A  - B\,)^D,
\label{w26}
\\
& \left[\, (AB)^k + (BA)^k  \, \right]^D   = (\, A + B \,)^D\left[(\,  A + B - I_m  \,)^D\right]^{2k-1} = \left[(\, A + B - I_m \,)^D\right]^{2k-1}(\, A  + B\,)^D,
\label{w27}
\\
& \left[ (ABA)^k - (BAB)^k \right]^D  = (\, A - B \,)^D \left[(\,  A + B - I_m  \,)^D \right]^{2k}  = \left[(\,  A + B - I_m  \,)^D \right]^{2k}
(\, A - B\,)^D,
\label{w28}
\\
& \left[ (ABA)^k + (BAB)^k \right]^D  = (\, A + B \,)^D \left[(\,  A
+ B - I_m  \,)^D \right]^{2k} = \left[(\,  A + B - I_m\,)^D\right]^{2k} (\, A + B\,)^D.
 \label{w29}
\end{align}
\end{corollary}

\begin{proof}
It follows from  \cite[Corollary 5]{Ti-pmd} that
\begin{align}
MN = \pm NM  \Rightarrow (MN)^D =N^DM^D.
\label{w30}
\end{align}
Applying this result to \eqref{w3}, \eqref{w4}, and \eqref{w10}--\eqref{w13} yields the results in this
corollary.
\end{proof}

\begin{remark}\label{TK36}
{\rm The Drazin inverses of $A \pm B$ for two idempotents $A$ and
$B$ were studied, and some formulas for $(\, A \pm  B\,)^D$ were
derived, see e.g.,  \cite{CD,Den1,Den2,HWW,ZW}. Substituting the
results in these papers into the equalities in Corollary \ref{TK35} may yield
a variety of equalities for the Drazin inverses of $AB \pm BA$ and
$(AB)^k \pm (BA)^k$, etc.
}
\end{remark}

A group of identities for $(AB)^k$, $(ABA)^k$, and $A - ABA$ are given
in the following theorem. Their verifications are straightforward.

\begin{theorem} \label{TK37}
Let $A$ and $B$ be two idempotent matrices of the order $m,$ and let $k$ be a positive integer$.$ Then$,$ the following matrix identities hold
\begin{align}
A - ABA & = A(\, A - B \,)^2 = (\, A - B \,)^2A,
\label{w32}
\\
B - BAB & = B(\, A - B \,)^2 = (\, A - B \,)^2B,
\label{w33}
\\
(\, A - ABA \,)^k & = A(\, A - B \,)^{2k} = (\, A - B \,)^{2k}A,
\label{w34}
\\
(\, B - BAB \,)^k & = B(\, A - B \,)^{2k} = (\, A - B \,)^{2k}B,
\label{w35}
\\
ABA & = A(\, A + B - I_m \,)^2 = (\, A + B -I_m \,)^2A,
\label{w36}
\\
BAB & = B(\, A +  B - I_m \,)^2 = (\, A + B  - I_m \,)^2B,
\label{w37}
\\
(ABA)^k & = A(\, A + B - I_m \,)^{2k} = (\, A + B -I_m \,)^{2k}A,
\label{w38}
\\
(BAB)^k & = B(\, A +  B - I_m \,)^{2k} = (\, A + B  - I_m \,)^{2k}B,
\label{w39}
\\
(BA)^2 & = BA(\, A + B - I_m \,)^2 = B(\, A + B -I_m \,)^2A,
\label{w40}
\\
(AB)^2 & = AB(\, A +  B - I_m \,)^2 = A(\, A + B  - I_m \,)^2B,
\label{w41}
\\
(AB)^k & = A(\, A + B - I_m \,)^{k}B,
\label{w42}
\\
(BA)^k & = B(\, A +  B - I_m \,)^{k}A.
\label{w43}
\end{align}
\end{theorem}

The following result is derived from Theorem \ref{TK37} and \eqref{w30}.

\begin{corollary} \label{TK38}
  Let $A$ and $B$ be two idempotent
matrices of the order $m,$ and let $k$ be a positive integer$.$ Then$,$ the following matrix identities hold
\begin{align}
(\, A - ABA \,)^D & = A[\,(\, A - B \,)^D\,]^2 =
[\,(\, A - B \,)^D\,]^2A,
\label{w44}
\\
(\, B - BAB \,)^D & = B[\,(\, A - B \,)^D\,]^2 =
[\,(\, A - B \,)^D\,]^2B,
\label{w45}
\\
(ABA)^D & = A[\,(\, A + B - I_m \,)^D\,]^2 =
[\,(\, A + B - I_m \,)^D\,]^2A,
\label{w46}
\\
(BAB)^D & = B[\,(\, A + B - I_m \,)^D\,]^2 =
[\,(\, A + B - I_m \,)^D\,]^2B.
\label{w47}
\end{align}
\end{corollary}

If $A$ and $B$ happen to be orthogonal projectors of the same size,
the results in the previous theorems  and corollaries are all valid.
In such cases, the matrices $A \pm B$ and $A + B - I_m$  are
Hermitian, and the results in the theorems  and corollaries  can be
simplified further. In particular, note that $M^{D} = M^{\dag}$ for
a Hermitian matrix. Hence  Corollaries  \ref{TK35} and \ref{TK38} reduce to the
following results.

\begin{corollary} \label{TK39}
Let $A$ and $B$ be two orthogonal projectors of the order $m,$
and let $k$ be a positive integer$.$  Then$,$ the following matrix identities hold
\begin{align}
 (\, AB - BA\,)^{\dag} & = -(\, A - B \,)^{\dag}(\, A + B - I_m \,)^{\dag} =
 (\, A + B - I_m \,)^{\dag}(\, A - B \,)^{\dag},
\label{w48}
\\
(\, AB + BA \,)^{\dag} & = (\, A + B \,)^{\dag}(\, A + B - I_m \,)^{\dag} =
(\, A + B - I_m \,)^{\dag}(\,  A + B\,)^{\dag},
\label{w49}
\\
\left[\, (AB)^k - (BA)^k  \,\right]^{\dag} & = (\, A - B \,)^{\dag}
\left[\, (\,  A + B - I_m  \,)^{\dag}\,\right]^{2k-1} = - \left[\, (\, A + B - I_m \,)^{\dag} \,\right]^{2k-1}(\, A  - B\,)^{\dag},
\label{w51}
\\
\left[\, (AB)^k + (BA)^k  \,\right]^{\dag} & = (\, A + B \,)^{\dag}\left[\,
(\,  A + B - I_m  \,)^{\dag} \,\right]^{2k-1}  = \left[\, (\, A + B - I_m \,)^{\dag}\,\right]^{2k-1}(\, A  +  B \,)^{\dag},
\label{w53}
\\
\left[\, (ABA)^k - (BAB)^k \,\right]^{\dag} & = (\, A - B \,)^{\dag}\left[\,
(\,  A + B - I_m  \,)^{\dag} \,\right]^{2k}  = \left[\,(\,  A + B - I_m  \,)^{\dag}  \,\right]^{2k}(\, A - B\,)^{\dag},
\label{w55}
\\
\left[\, (ABA)^k - (BAB)^k \,\right]^{\dag} & = (\, A + B \,)^{\dag}
\left[\,(\,  A + B - I_m  \,)^{\dag}\,\right]^{2k} = \left[\,(\,  A + B - I_m  \,)^{\dag} \,\right]^{2k}
(\, A + B\,)^{\dag},
\label{w57}
\\
(\, A - ABA \,)^{\dag} & = A\left[\,(\, A - B \,)^{\dag}\,\right]^2 =
\left[\,(\, A - B \,)^{\dag}\,\right]^2A,
\label{w58}
\\
(\, B - BAB \,)^{\dag} & = B\left[\,(\, A - B \,)^{\dag}\,\right]^2 =
\left[\,(\, A - B \,)^{\dag}\,\right]^2B,
\label{w59}
\\
(ABA)^{\dag} & = A\left[\,(\, A + B - I_m \,)^{\dag}\,\right]^2 =
\left[\,(\, A + B - I_m \,)^{\dag}\,\right]^2A,
\label{w60}
\\
(BAB)^{\dag} & = B\left[\,(\, A + B - I_m \,)^{\dag}\,\right]^2 =
\left[\,(\, A + B - I_m \,)^{\dag}\,\right]^2B.
\label{w61}
\end{align}
\end{corollary}

It was shown in \cite{CT-mp} that
\begin{align*}
(AB)^{\dag} & = BA - B[\,(\, I_m - B \,)(\, I_m - A \,)\,]^{\dag}A,
\\
(\, A - B \,)^{\dag} & = (\, A - AB \,)^{\dag} -
(\, B -  AB \,)^{\dag},
\\
(\, A - B \,)^{\dag} & = A - B + B(\, A - BA \,)^{\dag} -
(\, B -BA \,)^{\dag}A,
\\
( \, A + B - I_m \, )^{\dag} & = (AB)^{\dag} -
[\, (\, I_m - A \,)(\, I_m - B \, )\, ]^{\dag}
\end{align*}
hold for any pair of orthogonal projectors
$A$ and $B$. Substituting them into Corollary \ref{TK39} will yield many
new equalities for the Moore--Penrose inverses of orthogonal projectors.
More results and facts on the  Moore--Penrose inverses of two orthogonal
projectors and their algebraic operations can be found in \cite{Tian:2019b}.

In addition, it is also worth to discuss the following reverse order laws for generalized inverses
\begin{align*}
 (\alpha AB + \beta BA)^{(i,\ldots,j)} & = (\, A + B - I_m \,)^{(k,\ldots,l)}(\, \alpha A +  \beta B \,)^{-1}
\\
& = (\, \beta A +  \alpha B \,)^{-1}(\, A + B - I_m \,)^{(i,\ldots,j)},
\\
 (\alpha ABA + \beta BAB)^{(i,\ldots,j)} & = [(\, A + B - I_m \,)^{2}]^{(i,\ldots,j)}(\, \alpha A +  \beta B \,)^{(i,\ldots,j)}
\\
& =  (\, \beta A +  \alpha B \,)^{(i,\ldots,j)}[(\, A + B - I_m \,)^{2}]^{(i,\ldots,j)},
\\
 [\alpha (AB)^k + \beta (BA)^k]^{(i,\ldots,j)} & = [(\, A + B - I_m \,)^{2k-1}]^{(i,\ldots,j)}(\, \alpha A +  \beta B \,)^{(i,\ldots,j)}
\\
& = (\, \beta A +  \alpha B \,)^{(i,\ldots,j)}[(\, A + B - I_m \,)^{2k-1}]^{(i,\ldots,j)},
\\
[\alpha (ABA)^k + \beta (BAB)^k]^{(i,\ldots,j)} & = [(\, A + B - I_m \,)^{2k}]^{(i,\ldots,j)}(\, \alpha A +  \beta B \,)^{(i,\ldots,j)}
\\
& = (\, \beta A +  \alpha B \,)^{(i,\ldots,j)}[(\, A + B - I_m \,)^{2k}]^{(i,\ldots,j)}.
\end{align*}

We next explore ranges and null spaces of some matrix expressions composed by two idempotent matrices.
It is easy to verify that
\begin{align*}
& r[\, A + B, \, I_m + A - B\,] =  r[\, A + B, \, I_m - A + B\,] = r[\, A - B, \, I_m + A + B\,]  =  r[\, A - B, \, I_m - A - B\,] = m
\end{align*}
hold for two idempotent matrices $A$ and $B$ of the order $m$. Hence,
\begin{align*}
& r[\, A - B, \, I_m - A - B] =  r(A + B) +  r(I_m - A - B)
\\
& \Leftrightarrow r\!\begin{bmatrix} A \\ B \end{bmatrix} = r(A) + r(B) - r(AB) \ {\rm and} \
r[A, \, B] = r(A) + r(B) - r(BA)
\\
& \Leftrightarrow r\!\begin{bmatrix} A \\ B \end{bmatrix} = r(A) + r(B) - r(BA) \ {\rm and} \
r[A, \, B] = r(A) + r(B) - r(AB).
\end{align*}
In Problem 31-4, Issue 31(2003), the Bulletin of the International
Linear Algebra Society, the present author proposed such a problem: if $R$ is a
ring with unity 1, and $a, \ b \in R$ satisfy $a^2 = a$ and  $b^2 = b$, then
$$
(ab-ba)R = (a - b)R \cap  (1 - a - b)R \ \  {\rm and} \ \  R(ab-ba) =
  R(a - b) \cap R(1 - a - b).
$$
In matrix situation, the problem can be restated as follows: If $A$ and $B$
are a pair of idempotent matrices of order $m$, then
\begin{align}
{\mathscr R}(\,AB -  BA\,) = {\mathscr R}(\, A - B \,) \cap
  {\mathscr R}(\,I_m - A - B \,).
\label{w62}
\end{align}
It can be seen from this range equality that
\begin{enumerate}
\item[(i)] if $A - B$ is nonsingular, then ${\mathscr R}(\,AB -  BA\,) =
{\mathscr R}(\, I_m - A - B \,)$;

\item[(ii)]  if $A +  B - I_m$ is nonsingular, then
${\mathscr R}(\, AB -  BA\,) = {\mathscr R}(\, A - B \,)$;

\item[(iii)] the commutator $AB - BA$ is nonsingular if and only if both
$A - B$ and $I_m - A - B$ are nonsingular;

\item[(iv)] $AB = BA$  $\Leftrightarrow$ ${\mathscr R}(\, A - B \,) \cap
{\mathscr R}(\, I_m - A - B \,) = \{  0 \}$  $\Leftrightarrow$ $r\!\begin{bmatrix} A \\ B \end{bmatrix} = r(A) + r(B) - r(AB) \ {\rm and} \
r[A, \, B] = r(A) + r(B) - r(BA)$ $\Leftrightarrow$ $r\!\begin{bmatrix} A \\ B \end{bmatrix} = r(A) + r(B) - r(BA) \ {\rm and} \
r[A, \, B] = r(A) + r(B) - r(AB).$
  \end{enumerate}
The range equality in \eqref{w62} can be used to find some other
  analogous range equalities for matrix expressions consisting of
a pair of idempotent matrices of the same order.

\begin{theorem} \label{TK310}
Let $A$ and $B$ be two idempotent matrices of the order $m.$ Then$,$ the following range identities hold
\begin{align}
{\mathscr R}(\, AB + BA \,) & = {\mathscr R}(\, A + B \,) \cap
{\mathscr R}(\, A + B - I_m \,),
\label{w63}
\\
{\mathscr R}(\, ABA + BAB \,) & = {\mathscr R}(\, A + B \,) \cap
{\mathscr R}[\, (\, A + B - I_m \,)^2\,],
\label{w64}
\\
{\mathscr R}(\, ABA - BAB \,) & = {\mathscr R}(\, A - B \,) \cap
{\mathscr R}[\, (\, A + B - I_m \,)^2 \,],
\label{w65}
\\
{\mathscr R}[\, (\, AB - BA \,)^2 \,] & =
{\mathscr R}[\,(\, A - B \,)^2\,] \cap {\mathscr R}[\, (\, A + B -
  I_m \,)^2 \,].
\label{w66}
\end{align}
\end{theorem}

\begin{proof}
It can be derived from Lemma \ref{TK32} that
\begin{align}
{\mathscr R}(\, AB + BA \,) & \subseteq  {\mathscr R}(\, A + B \,) \cap
{\mathscr R}(\, A + B - I_m \,),
\label{w67}
\\
{\mathscr R}(\, ABA + BAB \,) & \subseteq  {\mathscr R}(\, A + B \,) \cap
{\mathscr R}[\, (\, A + B - I_m \,)^2 \,],
\label{w68}
\\
{\mathscr R}(\, ABA - BAB \,) & \subseteq {\mathscr R}(\, A - B \,) \cap
{\mathscr R}[\, (\, A + B - I_m \,)^2 \,],
\label{w69}
\\
{\mathscr R}[\, (\, AB - BA \,)^2 \,] & \subseteq
{\mathscr R}[\,(\, A - B \,)^2\,] \cap {\mathscr R}[\, (\, A + B -
  I_m \,)^2 \,].
\label{w70}
\end{align}
If we can show that the dimensions of the subspaces on both sides of
\eqref{w67}--\eqref{w70} are equal, respectively, then
\eqref{w67}--\eqref{w70} are valid.

It follows from \eqref{hh21} and \eqref{w3} that
\begin{align}
  r(\, AB + BA \,) = r(\,  A + B \,)
+ r(\, A + B - I_m \,) - m.
\label{w74}
\end{align}
On the other hand, using the well-known formula
\begin{align}
\dim [\, {\mathscr R}(M) \cap {\mathscr R}(N) \,] =
r(M) +  r(N) - r[\, M, \,  N \,],
\label{w75}
\end{align}
we also obtain
\begin{align}
&  \dim \left[\, {\mathscr R}(\, A + B \,) \cap
{\mathscr R} (\, A + B - I_m \,) \, \right] \nb
\\
& = r(\, A + B \,) +
  r(\, A + B - I_m \,) -  r[\,A + B, \,  A + B - I_m] \nb
\\
 & = r(\, A + B \,) + r(\, A + B - I_m \,) -  r[\, A + B, \, I_m ] \nb
\\
 & = r(\, A + B \,) + r(\, A + B - I_m \,) - m.
 \label{w76}
\end{align}
Eqs.\,\eqref{w74} and \eqref{w76} imply that the dimensions of the subspaces on
the both sides of \eqref{w67} are equal. Thus, \eqref{w63} holds. Note
from \eqref{w3} that
$$
ABA - BAB = (\,  A - B \,)(\, A + B - I_m \,)^2 =
(\,  A - B \,) - (\, A - B \,)^3.
$$
Hence, it follows from \eqref{hh22} that
\begin{align*}
  r(\, ABA - BAB \,) & = r(\, A - B \,) +
r[\, I_m - (\, A - B \,)^2 \,] - m
\\
& = r(\, A - B \,) +  r[\,(\, A + B  - I_m \,)^2 \,] - m
\\
& = r(\, A - B \,) +  r(I_m + A - B) +  r(I_m - A + B) - 2m  \ \ \ \mbox{(by \eqref{v34})},
\end{align*}
and by \eqref{w75},
\begin{align*}
&  \dim \{ \, {\mathscr R}(\, A - B \,) \cap
{\mathscr R}[\, (\, A + B - I_m \,)^2 \,] \,  \}
\\
&  = r(\, A - B \,) + r[\, (\, A + B - I_m \,)^2 \,]
  - r[\, A - B, \  (\, A + B - I_m \,)^2 \, ]
\\
&  =  r(\, A - B \,) + r[\, (\, A + B - I_m \,)^2 \,] -
r[\, A - B, \  I_m - (\, A - B \,)^2 \, ] \ \ \ \mbox{(by (\ref{v31}))}
\\
&  = r(\, A - B \,) + r[\, (\, A + B - I_m \,)^2 \,]
  - r[\, A - B, \  I_m  \, ]
\\
& =  r(\, A - B \,) + r[\, (\, A + B - I_m \,)^2 \,] - m
\\
& = r(\, A - B \,) +  r(I_m + A - B) +  r(I_m - A + B) - 2m \ \ \ \mbox{(by \eqref{v34})}.
\end{align*}
The above two equalities imply that the dimensions of the subspaces on
the both sides of \eqref{w69} are equal. Thus, \eqref{w67} holds.
Similarly, we can show that
\begin{align*}
& \dim [\, {\mathscr R}(\, ABA + BAB \,) \,]  =  r(\, A + B \,) +
  r[\,(\, A + B  - I_m \,)^2 \,] - m,
\\
& \dim\{ [\, {\mathscr R}(\, A + B \,) \cap {\mathscr R}[\,
(\, A + B - I_m \,)^2 \,] \}  =
  r(\, A + B \,) + r[\, (\, A + B - I_m \,)^2 \,] - m.
\end{align*}
Thus, \eqref{w64} holds. It can also be shown that
\begin{align*}
& \dim{\mathscr R}[\, (\, AB - BA \,)^2 \,]  =
r[\,(\, A - B \,)^2\,] + r[\, (\, A + B - I_m \,)^2 \,] - m,
\\
& \dim{\mathscr R}\{ \,[\,(\, A - B \,)^2\,] \cap {\mathscr R}[\, (\, A + B - I_m \,)^2 \,] \, \}
 = r[\,(\, A - B \,)^2\,] + r[\, (\, A + B - I_m \,)^2 \,]
- m.
\end{align*}
Hence, \eqref{w66} holds.
\end{proof}

We leave the verification of the following result to the reader.

\begin{theorem} \label{TK311}
Let $A$ and $B$ be two idempotent matrices of the order $m.$
Then
\begin{align*}
 {\mathscr N}(\, AB \pm BA \,) & = {\mathscr N}(\, A \pm B \,) \cap
{\mathscr N}(\, A + B - I_m \,),
\\
{\mathscr N}(\, ABA \pm BAB \,) & = {\mathscr N}(\, A \pm B \,)
\cap {\mathscr N}[\, (\, A + B - I_m \,)^2\,],
\\
{\mathscr N}[\, (\, AB - BA \,)^2 \,] & = {\mathscr N}[\,(\, A - B
\,)^2\,] \cap {\mathscr N}[\, (\, A + B - I_m \,)^2 \,].
\end{align*}
\end{theorem}

Obviously, the formulas and facts in Theorem \ref{TH313} can be extended to a family of idempotent matrices and their algebraic operations. It is expected that many more nontrivial algebraic equalities among idempotent matrices and their consequences can be established.

One of the main concerns about a given matrix is to explore the relationships between the matrix and its transpose or conjugate transpose; see e.g., \cite{Duk,Gor,Gow,Rad,TZa,Ver}. It is easy to see that $(A^{*})^2 = A^{*}$  by taking conjugate of both sides of $A^2 = A$.
In this case, replacing $B$ with $A^{*}$ in the preceding results, we obtain the following consequences.

\begin{theorem} \label{TL31}
Let $A$ be an idempotent matrix of the order $m.$  Then the following two matrix identities
\begin{align*}
&  (A - A^{*})^2 +  (A + A^{*} - I_m)^2  = I_m,
\\
& AA^{*} + A^{*}A + 4^{-1}I_m  = (\, A + A^{*} - 2^{-1}I_m  \,)^2,
\end{align*}
and the following five rank formulas
\begin{align*}
&  r(A - A^{*}) =  r(A + A^{*}) +  r(2I_m - A - A^{*}) - m,
\\
&  r(I_m - A - A^{*} ) =  2r(I_m + A - A^{*}) - m,
\\
& r(AA^{*} + A^{*}A) = r[A, \, A^{*}]  = r(I_m - A - A^{*}) + r(A + A^{*}) - m,
\\
& r(I_m - AA^{*} - A^{*}A) = r[(\sqrt{5} -1)/2I_m + A + A^{*}] + r[(\sqrt{5} + 1)/2I_m - A - A^{*}] - m,
\\
& r(2I_m - AA^{*} - A^{*}A) = r(I_m + A + A^{*} ) + r(2I_m - A - A^{*}) - m
\end{align*}
hold$.$ In particular$,$
\begin{align*}
&  A = A^{*} \Leftrightarrow  (I_m - A - A^{*} )^2  = I_m  \Leftrightarrow r(A + A^{*}) +  r(2I_m - A - A^{*}) = m,
\\
&  (A - A^{*})^2  = 2^{-1} I_m \Leftrightarrow  (I_m - A - A^{*} )^2  =  2^{-1}I_m,
\\
&  (A - A^{*})^2  = I_m \Leftrightarrow  (I_m - A - A^{*} )^2  = 0 \Leftrightarrow r(I_m + A - A^{*}) +  r(I_m - A + A^{*}) - m,
\\
& AA^{*} + A^{*}A =  - 2I_m    \Leftrightarrow  (\, A + A^{*} - 2^{-1}I_m  \,)^2 = - \frac{7}{4}I_m,
\\
& AA^{*} + A^{*}A =  -I_m    \Leftrightarrow  (\, A + A^{*} - 2^{-1}I_m  \,)^2 = - \frac{3}{4}I_m,
\\
& AA^{*} + A^{*}A =  -4^{-1}I_m  \Leftrightarrow (\, A + A^{*} - 2^{-1}I_m  \,)^2 =0,
\\
& AA^{*} + A^{*}A = 0  \Leftrightarrow  (\, A + A^{*} - 2^{-1}I_m  \,)^2 = 4^{-1}I_m \Leftrightarrow  r(I_m - A - A^{*} ) + r(A + A^{*}) = m,
\\
& AA^{*} + A^{*}A = \frac{3}{4}I_m   \Leftrightarrow  (\, A + A^{*} - 2^{-1}I_m  \,)^2 = I_m,
\\
& AA^{*} + A^{*}A = I_m   \Leftrightarrow  (\, A + A^{*} - 2^{-1}I_m  \,)^2 = \frac{5}{4}I_m   \Leftrightarrow r[(\sqrt{5} -1)/2I_m + A + A^{*}] + r[(\sqrt{5} + 1)/2I_m - A - A^{*}] = m,
\\
& AA^{*} + A^{*}A = 2I_m   \Leftrightarrow (\, A + A^{*} - 2^{-1}I_m  \,)^2 = \frac{9}{4}I_m  \Leftrightarrow  r(I_m + A + A^{*} ) + r(2I_m - A - A^{*}) = m,
\end{align*}
and
\begin{align*}
&  r(A - A^{*}) = m  \Leftrightarrow r(A + A^{*}) =  r(2I_m - A - A^{*}) = m,
\\
& r(I_m - A - A^{*}) = m  \Leftrightarrow  r(I_m + A - A^{*}) = m,
\\
& r(AA^{*} + A^{*}A) = r[A, \, A^{*}]  = m \Leftrightarrow r(A + A^{*}) = r(I_m - A - A^{*})  = m,
\\
& r(I_m - AA^{*} - A^{*}A) = m \Leftrightarrow r[(\sqrt{5} -1)/2I_m + A + A^{*}] =
r[(\sqrt{5} + 1)/2I_m - A - A^{*}] = m,
\\
& r(2I_m - AA^{*} - A^{*}A) = m \Leftrightarrow r(I_m + A + A^{*} ) = r(2I_m - A - A^{*}) = m.
\end{align*}
\end{theorem}

\begin{theorem}\label{TL32}
Let $A$ be an idempotent matrix of the order $m,$ and let $k$
  be any positive integer$.$ Then the following factorization equalities
\begin{align*}
\alpha AA^{*} + \beta A^{*}A & = (\, \alpha A +  \beta A^{*} \,)(\, A + A^{*} - I_m \,)
= (\, A + A^{*} - I_m \,) (\, \beta A +  \alpha A^{*} \,),
\\
\alpha AA^{*}A + \beta A^{*}AA^{*} & = (\, \alpha A +  \beta A^{*} \,)(\, A + A^{*} - I_m \,)^2
= (\, A + A^{*} - I_m \,)^2 (\, \beta A +  \alpha A^{*} \,),
\\
\alpha (AA^{*})^k + \beta (A^{*}A)^k & = (\, \alpha A +  \beta A^{*} \,)(\, A + A^{*} - I_m \,)^{2k-1}
= (\, A + A^{*} - I_m \,)^{2k-1}(\, \beta A +  \alpha A^{*} \,),
\\
\alpha (AA^{*}A)^k + \beta (A^{*}AA^{*})^k & = (\, \alpha A +  \beta A^{*} \,)(\, A + A^{*} - I_m \,)^{2k}
= (\, A + A^{*} - I_m \,)^{2k}(\, \beta A +  \alpha A^{*} \,)
\end{align*}
hold for any two scalars $\alpha$ and $\beta.$  In particular$,$ $\alpha AA^{*} + \beta A^{*}A$ is nonsingular $\Leftrightarrow$  $\alpha (AA^{*})^k + \beta (A^{*}A)^k$ is nonsingular $\alpha AA^{*}A + \beta A^{*}AA^{*}$ is nonsingular $\Leftrightarrow$ $\alpha (AA^{*})^k + \beta (A^{*}A)^k$ is nonsingular
$\Leftrightarrow$ $\alpha (AA^{*}A)^k + \beta (A^{*}AA^{*})^k$ is nonsingular $\Leftrightarrow$ both
$\alpha A + \beta A^{*}$ and $A + A^{*} - I_m$ are  nonsingular; in which cases$,$  the following equalities hold
\begin{align*}
(\alpha AA^{*} + \beta A^{*}A)^{-1} & = (\, A + A^{*} - I_m \,)^{-1}(\, \alpha A +  \beta A^{*} \,)^{-1}
= (\, \beta A +  \alpha A^{*} \,)^{-1}(\, A + A^{*} - I_m \,)^{-1},
\\
(\alpha AA^{*}A + \beta A^{*}AA^{*})^{-1} & = (\, A + A^{*} - I_m \,)^{-2}(\, \alpha A +  \beta A^{*} \,)^{-1}
=  (\, \beta A +  \alpha A^{*} \,)^{-1}(\, A + A^{*} - I_m \,)^{-2},
\\
[\alpha (AA^{*})^k + \beta (A^{*}A)^k]^{-1} & = (\, A + A^{*} - I_m \,)^{-2k+1}(\, \alpha A +  \beta A^{*} \,)^{-1}
= (\, \beta A +  \alpha A^{*} \,)^{-1}(\, A + A^{*} - I_m \,)^{-2k+1},
\\
[\alpha (AA^{*}A)^k + \beta (A^{*}AA^{*})^k]^{-1} & = (\, A + A^{*} - I_m \,)^{-2k}(\, \alpha A +  \beta A^{*} \,)^{-1}
= (\, \beta A +  \alpha A^{*} \,)^{-1}(\, A + A^{*} - I_m \,)^{-2k}.
\end{align*}
\end{theorem}

\begin{corollary}\label{TL33}
Let $A$  be an idempotent matrix of the order $m,$ and let $k$
  be a positive integer$.$ Then$,$ the following matrix identities hold
\begin{align*}
AA^{*} - A^{*}A & =  (\, A - A^{*} \,)(\, A + A^{*} - I_m \,) = - (\, A + A^{*} - I_m
\,)(\, A - A^{*} \,),
\\
AA^{*} + A^{*}A & = (\, A + A^{*} \,)(\, A + A^{*} - I_m  \,) =  (\, A + A^{*} -
I_m\,)(\, A + A^{*}\,),
\\
AA^{*}A - A^{*}AA^{*} & = (\, A - A^{*} \,)(\, A + A^{*} - I_m \,)^2 = (\, A + A^{*} - I_m \,)^2(\, A - A^{*} \,),
\\
AA^{*}A + A^{*}AA^{*} & = (\, A + A^{*} \,)(\, A + A^{*} - I_m \,)^2 =(\, A + A^{*} - I_m \,)^2(\, A + A^{*} \,).
\\
(\,AA^{*} - A^{*}A\,)^k  & = (-1)^{k(k-1)/2}(\, A - A^{*} \,)^k(\,  A + A^{*} - I_m \,)^k  = (-1)^{k(k-1)/2}(\, I_m - A - A^{*} \,)^k(\, A  - A^{*} \,)^k,
\\
(\,AA^{*} + A^{*}A \,)^k & = (\, A + A^{*}  \,)^k(\,  A + A^{*} - I_m \,)^k = (\,  A + A^{*} - I_m \,)^k(\, A + A^{*} \,)^k, \ \ \ \ \ \
\\
(\, AA^{*}A - A^{*}AA^{*}\,)^k & = (\,  A - A^{*} \,)^k(\, A + A^{*} - I_m \,)^{2k} = (\,  A + A^{*} - I_m\,)^{2k}(\,  A - A^{*}\,)^k,
\\
(\,AA^{*}A + A^{*}AA^{*}\,)^k & = (\, A + A^{*}  \,)^k(\,  A + A^{*} - I_m  \,)^{2k} = (\,  A + A^{*} - I_m \,)^{2k}(\,  A  +  A^{*}\,)^k,
\\
(AA^{*})^k - (A^{*}A)^k & = (\, A - A^{*} \,)(\,  A + A^{*} - I_m  \,)^{2k-1} = - (\, A + A^{*} - I_m \,)^{2k-1}(\, A  - A^{*}\,),
\\
(AA^{*})^k + (A^{*}A)^k & = (\, A + A^{*} \,)(\,  A + A^{*} - I_m  \,)^{2k-1} = (\, A + A^{*} - I_m \,)^{2k-1}(\, A + A^{*} \,),
\\
  (AA^{*}A)^k - (A^{*}AA^{*})^k & =  (\, A - A^{*} \,)(\,  A + A^{*} - I_m  \,)^{2k} = (\, A + A^{*} - I_m\,)^{2k}(\, A - A^{*}\,),
\\
(AA^{*}A)^k + (A^{*}AA^{*})^k & = (\, A + A^{*} \,)(\, A + A^{*} - I_m  \,)^{2k}  =(\,  A + A^{*} - I_m \,)^{2k}(\, A + A^{*}\,).
\end{align*}
In addition$,$ the following matrix identities hold
\begin{align*}
& AA^{*} - A^{*}A  + (AA^{*})^2 - (A^{*}A)^2 + \cdots + (AA^{*})^k - (A^{*}A)^k 
\\
& = (A - A^{*})[(\, A + A^{*} - I_m \,) + (\,  A + A^{*} - I_m  \,)^3  + \cdots +  (\,  A + A^{*} - I_m  \,)^{2k-1}] 
\\
& = [(\, A + A^{*} - I_m \,) + (\,  A + A^{*} - I_m  \,)^3  + \cdots +  (\,  A + A^{*} - I_m  \,)^{2k-1}](A^{*} - A),
\\
& AA^{*} + A^{*}A  + (AA^{*})^2 + (A^{*}A)^2 + \cdots + (AA^{*})^k + (A^{*}A)^k  
\\
& = (A + A^{*})[(\, A + A^{*} - I_m \,) + (\,  A + A^{*} - I_m  \,)^3  + \cdots +  (\,  A + A^{*} - I_m  \,)^{2k-1}]
\\
& = [(\, A + A^{*} - I_m \,) + (\,  A + A^{*} - I_m  \,)^3  + \cdots +  (\,  A + A^{*} - I_m  \,)^{2k-1}](A+A^{*}),
\\
& AA^{*}A - A^{*}AA^{*}  + (AA^{*}A)^2 - (A^{*}AA^{*})^2 + \cdots + (AA^{*}A)^k - (A^{*}AA^{*})^k  \nb
\\
& = (A - A^{*})[(\, A + A^{*} - I_m \,)^2 + (\,  A + A^{*} - I_m  \,)^4  + \cdots +  (\,  A + A^{*} - I_m  \,)^{2k}] \nb
\\
& = [(\, A + A^{*} - I_m \,)^2 + (\,  A + A^{*} - I_m  \,)^4  + \cdots +  (\,  A + A^{*} - I_m  \,)^{2k}](A - A^{*}),
\\
& AA^{*}A + A^{*}AA^{*}  + (AA^{*}A)^2 + (A^{*}AA^{*})^2 + \cdots + (AA^{*}A)^k + (A^{*}AA^{*})^k  \nb
\\
& = (A + A^{*})[(\, A + A^{*} - I_m \,)^2 + (\,  A + A^{*} - I_m  \,)^4  + \cdots +  (\,  A + A^{*} - I_m  \,)^{2k}] \nb
\\
& = [(\, A + A^{*} - I_m \,)^2 + (\,  A + A^{*} - I_m  \,)^4  + \cdots +  (\,  A + A^{*} - I_m  \,)^{2k}](A + A^{*}).
\end{align*}
\end{corollary}

\begin{theorem} \label{TL37}
Let $A$  be an idempotent matrices of the order $m,$ and let $k$ be a positive integer$.$ Then$,$ the following matrix identities hold
\begin{align*}
A - AA^{*}A & = A(\, A - A^{*} \,)^2 = (\, A - A^{*} \,)^2A,
\\
A^{*} - A^{*}AA^{*} & = A^{*}(\, A - A^{*} \,)^2 = (\, A - A^{*} \,)^2A^{*},
\\
(\, A - AA^{*}A \,)^k & = A(\, A - A^{*} \,)^{2k} = (\, A - A^{*} \,)^{2k}A,
\\
(\, A^{*} - A^{*}AA^{*} \,)^k & = A^{*}(\, A - A^{*} \,)^{2k} = (\, A - A^{*} \,)^{2k}A^{*},
\\
AA^{*}A & = A(\, A + A^{*} - I_m \,)^2 = (\, A + A^{*} -I_m \,)^2A,
\\
A^{*}AA^{*} & = A^{*}(\, A +  A^{*} - I_m \,)^2 = (\, A + A^{*}  - I_m \,)^2A^{*},
\\
(AA^{*}A)^k & = A(\, A + A^{*} - I_m \,)^{2k} = (\, A + A^{*} -I_m \,)^{2k}A,
\\
(A^{*}AA^{*})^k & = A^{*}(\, A +  A^{*} - I_m \,)^{2k} = (\, A + A^{*}  - I_m \,)^{2k}A^{*},
\\
(A^{*}A)^2 & = A^{*}A(\, A + A^{*} - I_m \,)^2 = A^{*}(\, A + A^{*} -I_m \,)^2A,
\\
(AA^{*})^2 & = AA^{*}(\, A +  A^{*} - I_m \,)^2 = A(\, A + A^{*}  - I_m \,)^2A^{*},
\\
(AA^{*})^k & = A(\, A + A^{*} - I_m \,)^{k}A^{*},
\\
(A^{*}A)^k & = A^{*}(\, A +  A^{*} - I_m \,)^{k}A.
\end{align*}
\end{theorem}

\begin{theorem} \label{TL319}
Let $A$ be an idempotent matrix of the order $m,$ and
let $k$ be a positive integer$.$ Then$,$ the following results hold$.$
\begin{enumerate}
\item[{\rm (a)}] ${\mathscr R}(AA^{*} - A^{*}A) \subseteq
{\mathscr R}(A - A^{*})$ and ${\mathscr R}(AA^{*} - A^{*}A)
\subseteq {\mathscr R}(A + A^{*} - I_m).$

\item[{\rm (b)}] ${\mathscr R}(AA^{*} + A^{*}A) \subseteq
{\mathscr R}(A + A^{*})$ and ${\mathscr R}(AA^{*} + A^{*}A)
\subseteq {\mathscr R}(A + A^{*} - I_m).$

\item[{\rm (c)}]  ${\mathscr N}(A - A^{*}) \subseteq {\mathscr
N}(AA^{*} - A^{*}A)$ and ${\mathscr N}(A + A^{*} -I_m)
\subseteq {\mathscr N}(AA^{*} - A^{*}A).$

\item[{\rm (d)}]  ${\mathscr N}(A + A^{*}) \subseteq {\mathscr
N}(AA^{*} + A^{*}A)$ and ${\mathscr N}(A + A^{*} -I_m)
\subseteq {\mathscr N}(AA^{*} + A^{*}A).$

\item[{\rm (e)}] ${\mathscr R}(\,AA^{*} \pm  A^{*}A\,) = {\mathscr R}(\, A \pm A^{*} \,) \cap
  {\mathscr R}(\, A + A^{*} - I_m \,).$

\item[{\rm (f)}] ${\mathscr R}(\, AA^{*}A \pm A^{*}AA^{*} \,)  = {\mathscr R}(\, A \pm A^{*} \,) \cap
{\mathscr R}(\, A + A^{*} - I_m \,).$

\item[{\rm (g)}] ${\mathscr N}(\, AA^{*} \pm A^{*}A \,)  = {\mathscr N}(\, A \pm A^{*} \,) \cap
{\mathscr N}(\, A + A^{*} - I_m \,).$

\item[{\rm (h)}] ${\mathscr N}(\, AA^{*}A \pm A^{*}AA^{*} \,)  = {\mathscr N}(\, A \pm A^{*} \,)
\cap {\mathscr N}(\, A + A^{*} - I_m \,).$

\end{enumerate}
\end{theorem}

\begin{theorem} \label{TL312}
Let $A$  be an idempotent matrices of the order $m.$ Then the following two identities
\begin{align}
(\, I_m +  \alpha A +  \beta A^{*} \,) & = (\, I_m +  \alpha A \,)
[\, I_m -(\alpha\beta)(1 + \alpha)^{-1}(1 + \beta)^{-1}AA^{*}\,](\, I_m + \beta A^{*} \,),
\label{pp391}
\\
(\, I_m +  \alpha A +  \beta A^{*} \,) & = (\, I_m + \beta A^{*}\,)[\, I_m -(\alpha\beta)(1 + \alpha)^{-1}(1 + \beta)^{-1}A^{*}A\,](\, I_m +  \alpha A \,)
\label{pp392}
\end{align}
hold for $\alpha \neq -1, 0$ and $\beta \neq -1, 0.$ In this case$,$ both $I_m +  \alpha A$ and $I_m + \beta A^{*}$ are nonsingular$.$ In particular$,$
$I_m - \lambda AA^{*}$ is nonsingular if and only if $I_m -  \alpha A  - \beta A^{*}$ is nonsingular$,$ in which case$,$
\begin{align}
& (\, I_m - \lambda AA^{*}\,)^{-1} = (\, I_m + \beta A^{*} \,)(\, I_m +  \alpha A  + \beta A^{*} \,)^{-1}(\, I_m  + \alpha A \,),
\label{pp393}
\\
& (\, I_m - \lambda A^{*}A\,)^{-1} = (\, I_m  + \alpha A \,)(\, I_m + \alpha A + \beta A^{*} \,)^{-1}(\,  I_m + \beta A^{*}\,)
\label{pp394}
\end{align}
hold$,$ where $\lambda = \alpha\beta(1 + \alpha)^{-1}(1 + \beta)^{-1}$.
\end{theorem}

In the remaining of this section, we present two groups of miscellaneous formulas for two/three idempotent matrices and their operations.

\begin{theorem} \label{TH312}
Let $A$ and $B$ be two idempotent matrices of the order $m.$ Then the following two identities
\begin{align}
(\, I_m +  \alpha A +  \beta B \,) & = (\, I_m +  \alpha A \,)
[\, I_m -(\alpha\beta)(1 + \alpha)^{-1}(1 + \beta)^{-1}AB\,](\, I_m + \beta B \,),
\label{391}
\\
(\, I_m +  \alpha A +  \beta B \,) & = (\, I_m + \beta B\,)[\, I_m -(\alpha\beta)(1 + \alpha)^{-1}(1 + \beta)^{-1}BA\,](\, I_m +  \alpha A \,)
\label{392}
\end{align}
hold for $\alpha \neq -1, 0$ and $\beta \neq -1, 0.$ In this case$,$ both $I_m +  \alpha A$ and $I_m + \beta B$ are nonsingular$.$ In particular$,$
$I_m - \lambda AB$ is nonsingular if and only if $I_m -  \alpha A  - \beta B$ is nonsingular$,$ in which case$,$
\begin{align}
& (\, I_m - \lambda AB\,)^{-1} = (\, I_m + \beta B \,)(\, I_m +  \alpha A  + \beta B \,)^{-1}(\, I_m  + \alpha A \,),
\label{393}
\\
& (\, I_m - \lambda BA\,)^{-1} = (\, I_m  + \alpha A \,)(\, I_m + \alpha A + \beta B \,)^{-1}(\,  I_m + \beta B\,)
\label{394}
\end{align}
hold$,$ where $\lambda = \alpha\beta(1 + \alpha)^{-1}(1 + \beta)^{-1}$.
\end{theorem}

\begin{proof}
Eqs.\,\eqref{391} and \eqref{392} were given in \cite{Vet}.  Eqs.\,\eqref{393} and \eqref{394} follow directly from \eqref{391} and \eqref{392}.
\end{proof}

It is also worth to discuss the following two families of reverse order laws for generalized inverses
\begin{align}
& (\, I_m - \lambda AB\,)^{(i,\ldots,j)} = (\, I_m + \beta B \,)(\, I_m +  \alpha A  + \beta B \,)^{(i,\ldots,j)} (\, I_m  + \alpha A \,),
\label{395}
\\
& (\, I_m  - \lambda BA\,)^{(i,\ldots,j)} = (\, I_m  + \alpha A \,)(\, I_m +  \alpha A  + \beta B \,)^{(i,\ldots,j)}(\,  I_m + \beta B\,)
\label{396}
\end{align}
associated with \eqref{391} and  \eqref{392}.

Concerning relationships among three idempotent matrices of the same size,  we have the following fundamental formulas and facts.

\begin{theorem} \label{TH313}
Let $A,$ $B,$ and $C$ be three idempotent matrices of the order $m,$ and denote $M = A + B + C.$ Then the following four identities
\begin{align}
& \alpha(AB + AC) + \beta (BA + BC) + \gamma(CA + CB)  = (\alpha A +  \beta B  +  \gamma C \,)(\,M - I_m),
\label{397}
\\
& \alpha(BA + CA) + \beta (AB + CB) + \gamma(AC + BC)  = (M - I_m)(\alpha A +  \beta B  +  \gamma C),
\label{398}
\end{align}
and
\begin{align}
&(\alpha + \beta)(AB + BA) + (\alpha + \gamma)(AC + CA) + (\beta + \gamma)(BC + CB) \nb
\\
& = (\alpha A +  \beta B  +  \gamma C)(\,M - I_m) + (M - I_m)(\alpha A +  \beta B  +  \gamma C),
\label{399}
\\
&(\alpha - \beta)(AB - BA) + (\alpha - \gamma)(AC - CA) + (\beta - \gamma)(BC - CB) \nb
\\
& = (\alpha A +  \beta B  +  \gamma C)M - M(\alpha A +  \beta B  +  \gamma C),
\label{3100}
\\
& \alpha(B + C)A(B + C)  + \beta (A + C)B(A + C) + \gamma(A + B)C(A + B) \nb
 \\
 & = (M - I_m)(\alpha A +  \beta B  +  \gamma C)(M - I_m)
\label{3100a}
\end{align}
hold for any three scalars $\alpha,$ $\beta,$ and $\gamma;$  the following identities hold
\begin{align}
&  (A + B)^2 +  (A + C)^2 +  (B + C)^2  = M(I_m + M),
\label{3106}
\\
&  (A - B)^2 +  (A - C)^2 +  (B - C)^2  = M(3I_m - M) = 9/4I_m - (M - 3/2I_m)^2,
\label{3107}
\\
&  AB + BA +  AC + CA +  BC + CB =  M(M - I_m) = (M - 2^{-1}I_m)^2 - 4^{-1}I_m,
\label{3108}
\\
& (AB + BA +  AC + CA +  BC + CB)^k =  M^k(M - I_m)^k, \ \ k = 1, 2, \ldots;
\label{3109}
\end{align}
the following rank formulas
\begin{align}
& r[(A + B)^2 +  (A + C)^2 +  (B + C)^2]  = r(M) + r(I_m + M) - m = \dim[{\mathscr R}(M) \cap {\mathscr R}(I_m + M)],
\label{3110}
\\
& r[(A - B)^2 +  (A - C)^2 +  (B - C)^2]  = r(M) + r(3I_m - M) - m = \dim[{\mathscr R}(M) \cap {\mathscr R}(3I_m - M)],
\label{3111}
\\
& r(kI_m - AB - BA -  AC - CA -  BC - CB) \nb
\\
& =  r\!\left[(\sqrt{4k +1} + 1)/2I_m - M \right] +
r\!\left[\,(\sqrt{4k +1} - 1)/2I_m + M \right] - m, \ \ k = 0, 1, \ldots, 6
\label{3112}
\end{align}
hold$;$ and the following range equalities hold
\begin{align}
& \R[(A + B)^2 +  (A + C)^2 +  (B + C)^2] =  \R(M)\cap \R(I_m + M),
\label{3112a}
\\
& \R[ (A - B)^2 +  (A - C)^2 +  (B - C)^2 ] =  \R(M)\cap \R(3I_m - M),
\label{3112b}
\\
& \R(AB + BA +  AC + CA +  BC + CB) =  \R(M)\cap \R(I_m - M),
\label{3112c}
\end{align}
and
\begin{align}
& \N[(A + B)^2 +  (A + C)^2 +  (B + C)^2] =  \R(N)\cap \N(I_m + M),
\label{3112d}
\\
& \N[ (A - B)^2 +  (A - C)^2 +  (B - C)^2 ] =  \N(M)\cap \N(3I_m - M),
\label{3112e}
\\
& \N(AB + BA +  AC + CA +  BC + CB) =  \N(M)\cap \N(I_m - M).
\label{3112f}
\end{align}
In particular$,$
\begin{enumerate}
\item[{\rm (a)}] the following facts hold
\begin{align*}
& \alpha(AB + AC) + \beta (BA + BC) + \gamma(CA + CB)  = 0   \Leftrightarrow (\alpha A +  \beta B  +  \gamma C)(M - I_m) = 0,
\\
& \alpha(BA + CA) + \beta (AB + CB) + \gamma(AC + BC)  = 0   \Leftrightarrow  (M - I_m)(\alpha A +  \beta B  +  \gamma C) = 0,
\end{align*}
and
\begin{align*}
& (\alpha + \beta)(AB + BA) + (\alpha + \gamma)(AC + CA) + (\beta + \gamma)(BC + CB)   =0   \nb
 \\
& \Leftrightarrow  (\alpha A +  \beta B  +  \gamma C \,)(\,M - I_m) + (M - I_m)(\alpha A +  \beta B  +  \gamma C) =0,
\\
&(\alpha - \beta)(AB - BA) + (\alpha - \gamma)(AC - CA) + (\beta - \gamma)(BC - CB) =0 \nb
\\
&   \Leftrightarrow (\alpha A +  \beta B  +  \gamma C \,)M = M(\alpha A +  \beta B  +  \gamma C),
\\
& \alpha(B + C)A(B + C)  + \beta (A + C)B(A + C) + \gamma(A + B)C(A + B) =0  \nb
\\
&\Leftrightarrow  (M - I_m)(\alpha A +  \beta B  +  \gamma C)(M - I_m) =0;
\end{align*}

\item[{\rm (b)}]  the following facts hold
\begin{align*}
& (A + B)^2 +  (A + C)^2 +  (B + C)^2  = 0  \Leftrightarrow  M^2 + M =0,
\\
& (A + B)^2 +  (A + C)^2 +  (B + C)^2  = I_m  \Leftrightarrow  M^2 + M = I_m,
\\
& (A - B)^2 +  (A - C)^2 +  (B - C)^2  = 0  \Leftrightarrow  (2M - 3I_m)^2  = 9I_m,
\\
& (A - B)^2 +  (A - C)^2 +  (B - C)^2  = 9/8I_m  \Leftrightarrow  (2M - 3I_m)^2  =  9I_m,
\\
& (A - B)^2 +  (A - C)^2 +  (B - C)^2  = 3I_m  \Leftrightarrow  (2M - 3I_m)^2  =  -3I_m,
\\
& (A - B)^2 +  (A - C)^2 +  (B - C)^2  = 9/4I_m  \Leftrightarrow  (2M - 3I_m)^2  = 0,
\\
&  AB + BA +  AC + CA +  BC + CB = kI_m  \Leftrightarrow (I_m - 2M)^2 = (4k +1)I_m,
\ \ k = 0, 1, \ldots, 6;
\end{align*}

\item[{\rm (c)}] $\alpha(AB + AC) + \beta (BA + BC) + \gamma(CA + CB)$  is nonsingular $\Leftrightarrow$  $\alpha(BA + CA) + \beta (AB + CB) + \gamma(AC + BC)$ is nonsingular  $\Leftrightarrow$ $\alpha(B + C)A(B + C)  + \beta (A + C)B(A + C) + \gamma(A + B)C(A + B)$ $\Leftrightarrow$ both $\alpha A + \beta B  + \gamma C$ and  $M - I_m$ are nonsingular; in which cases$,$ the following equalities hold
\begin{align*}
& [\alpha(AB + AC) + \beta (BA + BC) + \gamma(CA + CB)]^{-1}  = (M - I_m)^{-1}(\, \alpha A +  \beta B  +  \gamma C \,)^{-1},
\\
& [\alpha(BA + CA) + \beta (AB + CB) + \gamma(AC + BC)]^{-1}  = (\,\alpha A +  \beta B  +  \gamma C \,)^{-1},
(M - I_m)^{-1},
\end{align*}
and
\begin{align*}
& [\alpha(B + C)A(B + C)  + \beta (A + C)B(A + C) + \gamma(A + B)C(A + B)]^{-1}
\\
& = (M - I_m)^{-1}(\alpha A +  \beta B  +  \gamma C)^{-1}(M - I_m)^{-1};
\end{align*}

\item[{\rm (d)}] the following facts hold
\begin{align*}
& r[(A + B)^2 +  (A + C)^2 +  (B + C)^2]  = m  \Leftrightarrow r(M) = r(I_m + M) = m,
\\
& r[(A - B)^2 +  (A - C)^2 +  (B - C)^2]  = m  \Leftrightarrow r(M) = r(3I_m - M) = m,
\end{align*}
and
\begin{align*}
& r(kI_m - AB - BA -  AC - CA -  BC - CB) =m \nb
\\
 & \Leftrightarrow   r\!\left[(\sqrt{4k +1} + 1)/2I_m - M \right] =
r\!\left[\,(\sqrt{4k +1} - 1)/2I_m + M \right] = m, \ \ k = 0, 1, \ldots, 6.
\end{align*}
\end{enumerate}
\end{theorem}

\begin{proof}
Eqs.\,\eqref{397}, \eqref{398},  and \eqref{3100a} follow from expanding both sides of the three equalities. The sum and difference of \eqref{397} and \eqref{398} result in \eqref{399} and \eqref{3100}, respectively. Eqs.\,\eqref{3106}--\eqref{3109} follow from direct expansions.
Applying \eqref{hh21} to \eqref{3106}--\eqref{3108} leads to \eqref{3110}--\eqref{3112}.
Results (a)--(d) follow from \eqref{397}--\eqref{3112}.
\end{proof}

In addition, it would be of interest to discuss the following reverse order laws for generalized inverses
\begin{align*}
& [\alpha(AB + AC) + \beta (BA + BC) + \gamma(CA + CB)]^{(i,\ldots,j)}  = (M - I_m)^{(i,\ldots,j)}(\, \alpha A +  \beta B  +  \gamma C \,)^{(i,\ldots,j)},
\\
& [\alpha(BA + CA) + \beta (AB + CB) + \gamma(AC + BC)]^{(i,\ldots,j)}  = (\,\alpha A +  \beta B  +  \gamma C \,)^{(i,\ldots,j)},
(M - I_m)^{(i,\ldots,j)},
\end{align*}
and
\begin{align*}
& [\alpha(B + C)A(B + C)  + \beta (A + C)B(A + C) + \gamma(A + B)C(A + B)]^{(i,\ldots,j)}
\\
& = (M - I_m)^{(i,\ldots,j)}(\alpha A +  \beta B  +  \gamma C)^{(i,\ldots,j)}(M - I_m)^{(i,\ldots,j)}.
\end{align*}

\section[4]{Bounds of ranks of some matrix pencils composed by two idempotent matrices}

For a given singular matrix $A$, the two products $AA^{-}$ and $A^{-}A$ are not necessarily unique, but they both
are idempotent matrices for all $A^-$. In this case, people are interested in the performance of
idempotents associated generalized inverses and their operations; see e.g., \cite{BR,CX,DW,ND,T0}. In this section, we approach the ranks of the following four characteristic matrices
 \begin{align}
 \lambda I_m + AA^{-} \pm  BB^{-}, \ \  \lambda I_m + AA^{-}\pm C^{-}C
\label{d11}
\end{align}
associated with the idempotent matrices $A^{-}$, $B^{-}$, and $C^{-}$, where $\lambda$ is a scalar. It is obvious that the ranks of the four matrix expressions in
\eqref{d11} are all functions of $\lambda$, $A^{-}$, $B^{-}$, and $C^{-}$. Thus we are interested in the maximum and minimum ranks, as well as rank
distributions of the four matrix expressions. To determine the two ranks, we need to use the following known formulas.

\begin{lemma}[\cite{MS:1974}] \label{TN41}
Let $ A \in {\mathbb C}^{m \times n}, \, B \in {\mathbb C}^{m \times k},$ $C \in {\mathbb C}^{l \times n},$ and
$D \in {\mathbb C}^{l \times k}.$ Then
\begin{align}
r[A, \, B]  =   r(A) + r(E_AB), \ \ \ r\!\begin{bmatrix}  A  \\ C  \end{bmatrix} = r(A) +r(CF_A).
\label{k42}
\end{align}
\end{lemma}

\begin{lemma} [\cite{T1}] \label{TN43}
The maximum and minimum ranks of the linear matrix-valued function $A - B_1X_1C_1 -B_2X_2C_2$ with respect to the two variable matrices $X_1$ and $X_2$ are given by
\begin{align}
&  \max_{X_1, \, X_2}\!r(\, A - B_1X_1C_1 -B_2X_2C_2 \,)  = \min\!\left\{ r[\, A, \, B_1, \, B_2 \,], \
 r\!\begin{bmatrix} A  \\ C_1 \\ C_2 \end{bmatrix}\!, \ r\!\begin{bmatrix}  A
& B_1\\ C_2   & 0 \end{bmatrix}\!, \
r\!\begin{bmatrix} A & B_2 \\ C_1  & 0 \end{bmatrix} \right\},
\label{110}
\\
& \min_{X_1, \, X_2}\!r(\, A - B_1X_1C_1 -B_2X_2C_2 \,) = r\!\begin{bmatrix}  A  \\ C_1 \\ C_2 \end{bmatrix} + r[ \, A, \, B_1, \,  B_2 \,] \nb
 \\
& \  + \max\left\{\! r\!\begin{bmatrix}A & B_1\\ C_2 & 0\end{bmatrix} - r\!\begin{bmatrix} A
& B_1 & B_2  \\ C_2   & 0 & 0 \end{bmatrix}
-r\begin{bmatrix} A
& B_1 \\ C_1   & 0 \\ C_2 & 0 \end{bmatrix}\!, \   r\!\begin{bmatrix}  A
& B_2 \\ C_1 & 0 \end{bmatrix}  -r\!\begin{bmatrix} A
& B_1 & B_2  \\ C_1   & 0 & 0 \end{bmatrix} - r\!\begin{bmatrix} A
& B_2 \\ C_1   & 0 \\ C_2 & 0  \end{bmatrix} \!\right\}\!.
\label{111}
\end{align}
\end{lemma}

We next establish exact expansion formulas for calculating the maximum and minimum ranks of the four matrix pencils with respect to the choice of $\lambda$, $A^{-}$, and $B^{-}$, and use them to derive a variety of simple and interesting properties of the four matrix pencils from the rank formulas.

\begin{lemma} [\cite{BG,CM,RM}] \label{TN42}
Let $A \in {\mathbb C}^{m \times n}.$ Then the general expressions of $A^{-},$ $AA^{-},$ and $A^{-}A$ can be written as
\begin{align}
& A^{-} = A^{\dag} + F_{A}U + VE_{A},  \ \  AA^{-} = AA^{\dag} + AVE_{A}, \ \  A^{-}A = A^{\dag}A + F_{A}UA,
\label{112}
\end{align}
where $U,\, V \in {\mathbb C}^{n \times m}$ are arbitrary$.$
\end{lemma}

\begin{theorem} \label{TN44}
Let $A \in {\mathbb C}^{m \times n}$  and $B \in {\mathbb C}^{m \times p}$ be given$.$
\begin{enumerate}
\item[{\rm (a)}] If $\lambda \neq 0, -1, -2,$  then the following two formulas hold
\begin{align}
& \max_{A^{-},B^{-}}\!\!r(\, \lambda I_m + AA^{-} + BB^{-}\,) =  m,
 \label{z110}
\\
& \min_{A^{-},B^{-}}\!\!r(\,\lambda I_m + AA^{-} + BB^{-}\,) = \max\{ \, m + r(A) - r[\, A, \, B \,], \ \
 m + r(B) - r[\, A, \, B \,] \,\}.
\label{z111}
\end{align}
In particular$,$ the following results hold$.$
\begin{enumerate}
\item[{\rm (i)}] There always exist $A^{-}$ and  $B^{-}$ such that $\lambda I_m + AA^{-} + BB^{-}$ is nonsingular$.$

\item[{\rm (ii)}] $\lambda I_m + AA^{-} + BB^{-}$ is  nonsingular for all $A^{-}$ and $B^{-}$ $\Leftrightarrow$
The rank of $\lambda I_m + AA^{-} + BB^{-}$ is invariant for all $A^{-}$ and $B^{-}$ $\Leftrightarrow$
 $r[\, A, \, B \,] = r(A) = r(B)$ $\Leftrightarrow$ $\R(A) = \R(B).$

\item[{\rm (iii)}] There do not exist $A^{-}$ and  $B^{-}$ such that $\lambda I_m + AA^{-} + BB^{-} =0.$
\end{enumerate}

\item[{\rm (b)}] The following two formulas hold
\begin{align}
& \max_{A^{-},B^{-}}\!\!r(\,AA^{-} + BB^{-}\,) =  r[\, A, \, B \,],
 \label{z112}
\\
& \min_{A^{-},B^{-}}\!\!r(\,AA^{-} + BB^{-}\,) = \max\{ \, r(A), \ \  r(B) \,\}.
\label{z113}
\end{align}
In particular$,$ the following results hold$.$
\begin{enumerate}
\item[{\rm (i)}] There exist $A^{-}$ and  $B^{-}$ such that $AA^{-} + BB^{-}$ is nonsingular $\Leftrightarrow$ $r[\, A, \, B \,] =m.$

\item[{\rm (ii)}] $AA^{-} + BB^{-}$ is nonsingular for all $A^{-}$ and $B^{-}$ $\Leftrightarrow$  $r(A) =m$ or $r(B) =m.$

\item[{\rm (iii)}] There exist $A^{-}$ and  $B^{-}$ such that $AA^{-} + BB^{-} = 0$ $\Leftrightarrow$ $AA^{-} + BB^{-} =0$ for all $A^{-}$ and $B^{-}$ $\Leftrightarrow$ $[\, A, \, B \,] = 0.$

\item[{\rm (iv)}] The rank of $AA^{-} + BB^{-}$ is invariant for all $A^{-}$ and $B^{-}$ $\Leftrightarrow$  $\R(A) \supseteq \R(B)$ or $\R(A) \subseteq \R(B).$
\end{enumerate}

\item[{\rm (c)}] The following two formulas hold
\begin{align}
& \max_{A^{-},B^{-}}\!\!r(\, -I_m + AA^{-} + BB^{-}\,) =  m - |r(A) - r(B)|,
 \label{z114}
\\
& \min_{A^{-},B^{-}}\!\!r(\,-I_m + AA^{-} + BB^{-}\,) = m + r(A) + r(B) - 2r[\, A, \, B \,].
\label{z115}
\end{align}
In particular$,$ the following results hold$.$
\begin{enumerate}
\item[{\rm (i)}]  There exist $A^{-}$ and $B^{-}$ such that $- I_m + AA^{-} + BB^{-}$ is nonsingular $\Leftrightarrow$
$r(A) = r(B).$

\item[{\rm (ii)}] $ - I_m + AA^{-} + BB^{-}$ is nonsingular for all $A^{-}$ and $B^{-}$ $\Leftrightarrow$
$\R(A) =\R(B).$

\item[{\rm (iii)}] There exist $A^{-}$ and  $B^{-}$ such that $AA^{-} + BB^{-} = I_m$ $\Leftrightarrow$  $r[\, A, \, B \,] = r(A) + r(B) =m.$

\item[{\rm (iv)}] The rank of $- I_m + AA^{-} + BB^{-}$ is invariant for all $A^{-}$ and $B^{-}$ $\Leftrightarrow$ $\R(A) \supseteq \R(B)$ or $\R(A) \subseteq \R(B).$
\end{enumerate}

\item[{\rm (d)}] The following two formulas hold
\begin{align}
& \max_{A^{-},B^{-}}\!\!r(\, -2I_m + AA^{-} + BB^{-}\,) =  m + r[\, A, \, B \,] - r(A) - r(B),
 \label{z116}
\\
& \min_{A^{-},B^{-}}\!\!r(\,-2I_m + AA^{-} + BB^{-}\,) =  \max\{ \, m - r(A), \ \  m - r(B) \,\}.
\label{z117}
\end{align}
In particular$,$ the following results hold$.$
\begin{enumerate}
\item[{\rm (i)}]  There exist $A^{-}$ and $B^{-}$ such that $- 2I_m + AA^{-} + BB^{-}$ is nonsingular $\Leftrightarrow$
$\R(A) \cap \R(B) =\{0\}.$

\item[{\rm (ii)}] $- 2I_m + AA^{-} + BB^{-}$ is nonsingular for all $A^{-}$ and $B^{-}$ $\Leftrightarrow$
 $A =0$ or $B =0.$

\item[{\rm (iii)}] There exist $A^{-}$ and $B^{-}$ such that $AA^{-} + BB^{-} = 2I_m$
$\Leftrightarrow$ $AA^{-} + BB^{-} =2I_m$ for all $A^{-}$ and $B^{-}$ $\Leftrightarrow$ $r(A) = r(B) =m.$

\item[{\rm (iv)}] The rank of $- 2I_m AA^{-} + BB^{-}$ is invariant for all $A^{-}$ and $B^{-}$ $\Leftrightarrow$ $\R(A) \supseteq \R(B)$ or $\R(A) \subseteq \R(B).$
\end{enumerate}
\end{enumerate}
\end{theorem}

\begin{proof}
By \eqref{112},
\begin{align}
&\lambda I_m + AA^{-} + BB^{-}  = \lambda I_m + AA^{\dag} + BB^{\dag} + AV_1E_{A} + BV_2E_{B},
 \label{z118}
\end{align}
where $V_1\in {\mathbb C}^{n \times m}$ and $V_2\in {\mathbb C}^{p\times m}$ are arbitrary$.$ Applying \eqref{110} and \eqref{111} to
\eqref{z118} and simplifying gives
\begin{align}
&  \max_{V_1, \, V_2} r(\, \lambda I_m + AA^{\dag} + BB^{\dag} + AV_1E_{A} + BV_2E_{B} \,) \nb
 \\
& = \min\!\left\{ r[\, \lambda I_m + AA^{\dag} + BB^{\dag}, \, A, \, B \,], \
 r\!\begin{bmatrix} \lambda I_m + AA^{\dag} + BB^{\dag}  \\ E_{A} \\ E_{B} \end{bmatrix}\!,
\right. \nb
  \\
&  \ \ \ \ \ \  \ \ \ \  \left.  r\!\begin{bmatrix}  \lambda I_m + AA^{\dag} + BB^{\dag}
& A\\ E_{B}   & 0 \end{bmatrix}\!, \
r\!\begin{bmatrix} \lambda I_m + AA^{\dag} + BB^{\dag} & B \\ E_{A}  & 0 \end{bmatrix} \right\} \nb
\\
& = \min\!\left\{ r[\, \lambda I_m, \, A, \, B \,], \ \
 r\!\begin{bmatrix} (\lambda +2)I_m   \\ E_{A} \\ E_{B} \end{bmatrix}\!, \ \   r\!\begin{bmatrix}  (\lambda +1)I_m & A \\ E_{B} & 0 \end{bmatrix}\!, \ \
r\!\begin{bmatrix} (\lambda +1)I_m  & B \\ E_{A}  & 0 \end{bmatrix} \right\},
\label{z119}
\end{align}
and
\begin{align}
& \min_{V_1, \, V_2}r(\, \lambda I_m + AA^{\dag} + BB^{\dag} + AV_1E_{A} + BV_2E_{B} \,) \nb
\\
& = r\!\begin{bmatrix}  \lambda I_m + AA^{\dag} + BB^{\dag}  \\ E_{A} \\ E_{B} \end{bmatrix} + r[ \, \lambda I_m + AA^{\dag} + BB^{\dag}, \, A, \,  B \,] \nb
 \\
& \ \  + \ \max\left\{\! r\!\begin{bmatrix}\lambda I_m + AA^{\dag} + BB^{\dag} & A\\ E_{B} & 0\end{bmatrix} - r\!\begin{bmatrix} \lambda I_m + AA^{\dag} + BB^{\dag}
& A & B  \\ E_{B}   & 0 & 0 \end{bmatrix}
-r\!\begin{bmatrix} \lambda I_m + AA^{\dag} + BB^{\dag}
& A \\ E_{A}   & 0 \\ E_{B} & 0 \end{bmatrix},  \right. \nb
\\
& \ \ \ \ \ \ \ \ \ \ \ \ \ \ \  \left.  r\!\begin{bmatrix}  \lambda I_m + AA^{\dag} + BB^{\dag}
& B \\ E_{A} & 0 \end{bmatrix}  -r\!\begin{bmatrix} \lambda I_m + AA^{\dag} + BB^{\dag}
& A & B  \\ E_{A}   & 0 & 0 \end{bmatrix} - r\!\begin{bmatrix} \lambda I_m + AA^{\dag} + BB^{\dag}
& B \\ E_{A}   & 0 \\ E_{B} & 0  \end{bmatrix} \!\right\} \nb
\\
&  =  r\!\begin{bmatrix} (\lambda +2)I_m   \\ E_{A} \\ E_{B} \end{bmatrix} + r[\, \lambda I_m, \, A, \, B \,] + \ \max\left\{\! r\!\begin{bmatrix}  (\lambda +1)I_m & A \\ E_{B} & 0 \end{bmatrix} - r\!\begin{bmatrix} \lambda I_m & A & B  \\ E_{B}   & 0 & 0 \end{bmatrix}
-r\!\begin{bmatrix} (\lambda +1)I_m
& A \\ E_{A}   & 0 \\ E_{B} & 0 \end{bmatrix},  \right. \nb
\\
& \ \ \ \ \ \ \ \ \ \ \ \ \ \ \  \left.
r\!\begin{bmatrix} (\lambda +1)I_m  & B \\ E_{A}  & 0 \end{bmatrix}  -r\!\begin{bmatrix} \lambda I_m
& A & B  \\ E_{A}   & 0 & 0 \end{bmatrix} - r\!\begin{bmatrix} (\lambda +1)I_m
& B \\ E_{A}   & 0 \\ E_{B} & 0  \end{bmatrix} \!\right\}.
\label{z120}
\end{align}
Substituting different values of $\lambda$ into the above formulas \eqref{z119} and \eqref{z120}
and simplifying yield the rank formulas in \eqref{z110}--\eqref{z117}, respectively.  The facts in (a)--(d) follow from setting the rank formulas  in \eqref{z110}--\eqref{z117}  equal to $m$ and 0 and applying Lemma \ref{TN41}, respectively.
\end{proof}

\begin{theorem} \label{TN45}
Let $A \in {\mathbb C}^{m \times n}$  and $B \in {\mathbb C}^{m \times p}$ be given$.$
\begin{enumerate}
\item[{\rm (a)}] If $\lambda \neq 1, 0, -1,$  then the following two formulas hold
\begin{align}
& \max_{A^{-},B^{-}}\!\!r(\,\lambda I_m + AA^{-} - BB^{-}\,) =m,
\label{z121}
\\
& \min_{A^{-},B^{-}}\!\!r(\,\lambda I_m + AA^{-} - BB^{-}\,) =  \max\{ \, m + r(A) - r[\, A, \, B \,], \ \ m + r(B) - r[\, A, \, B \,] \}.
\label{z122}
\end{align}
In particular$,$ the following results hold$.$
\begin{enumerate}
\item[{\rm (i)}] There always exist $A^{-}$ and  $B^{-}$ such that $\lambda I_m + AA^{-} - BB^{-}$ is nonsingular$.$

\item[{\rm (ii)}]  $\lambda I_m + AA^{-} - BB^{-}$ is  nonsingular for all $A^{-}$ and $B^{-}$ $\Leftrightarrow$
the rank of $\lambda I_m + AA^{-} - BB^{-}$ is invariant for all $A^{-}$ and $B^{-}$ $\Leftrightarrow$
 $r[\, A, \, B \,] = r(A) = r(B)$ $\Leftrightarrow$ $\R(A) = \R(B).$

\item[{\rm (iii)}] There do not exist $A^{-}$ and $B^{-}$ such that $\lambda I_m + AA^{-} - BB^{-} =0.$
\end{enumerate}

\item[{\rm (b)}] The following two formulas hold
\begin{align}
& \max_{A^{-},B^{-}}\!\!r(\,I_m + AA^{-} - BB^{-}\,) =\min \{m, \ \  m + r(A) - r(B) \},
\label{z123}
\\
& \min_{A^{-},B^{-}}\!\!r(\,I_m + AA^{-} - BB^{-}\,) =  m + r(A) - r[\, A, \, B \,].
\label{z124}
\end{align}
In particular$,$ the following results hold$.$
\begin{enumerate}
\item[{\rm (i)}] There exist $A^{-}$ and  $B^{-}$ such that $I_m + AA^{-} - BB^{-}$ is nonsingular
$\Leftrightarrow$ $r(A) = r(B).$

\item[{\rm (ii)}]  $I_m + AA^{-} - BB^{-}$ is  nonsingular for all $A^{-}$ and $B^{-}$ $\Leftrightarrow$
 $\Leftrightarrow$ $\R(A) \supseteq \R(B).$

\item[{\rm (iii)}] There exist $A^{-}$ and  $B^{-}$ such that $BB^{-} - AA^{-} =I_m$ $\Leftrightarrow$
$BB^{-} - AA^{-} =I_m$ holds for all $A^{-}$ and $B^{-}$  $\Leftrightarrow$ $A = 0$ and $r(B) =m.$

\item[{\rm (iv)}] The rank of $I_m + AA^{-} - BB^{-}$ is invariant for all $A^{-}$ and $B^{-}$
$\Leftrightarrow$ $\R(A) \supseteq \R(B)$ or $\R(A) \subseteq \R(B).$

\end{enumerate}

\item[{\rm (c)}] The following two formulas hold
\begin{align}
& \max_{A^{-},B^{-}}\!\!r(\,AA^{-} - BB^{-}\,) =   \min\{ \,  r[\, A, \, B \,], \ \  m + r[\, A, \, B \,] -  r(A) - r(B) \},
\label{z125}
\\
& \min_{A^{-},B^{-}}r(\,AA^{-} - BB^{-}\,) =  \max\{ \,  r[\, A, \, B \,] - r(A), \ \ r[\, A, \, B \,] - r(B) \}.
\label{z126}
\end{align}
In particular$,$ the following results hold$.$
\begin{enumerate}
\item[{\rm (i)}] There exist $A^{-}$ and $B^{-}$ such that $AA^{-} - BB^{-}$ is nonsingular
$\Leftrightarrow$   $r[\, A, \, B \,]  = r(A) + r(B)  =m.$

\item[{\rm (ii)}] $AA^{-} - BB^{-}$ is nonsingular for all $A^{-}$ and $B^{-}$ $\Leftrightarrow$ $r(A) =m$ and $B = 0,$
or $A = 0$ and $r(B) =m.$

\item[{\rm (iii)}] There exist $A^{-}$ and  $B^{-}$ such that $AA^{-} = BB^{-}$ $\Leftrightarrow$ $\R(A) =\R(B).$

\item[{\rm (iv)}]  $AA^{-} = BB^{-}$ holds for all $A^{-}$ and $B^{-}$ $\Leftrightarrow$ $[\, A, \, B \,] = 0$ or $r[\, A, \, B \,] = r(A) + r(B) -m.$

\item[{\rm (v)}] The rank of $AA^{-} - BB^{-}$ is invariant for all $A^{-}$ and $B^{-}$
$\Leftrightarrow$ $A=0$ or $B=0$  $r(A)=m$ or $r(B)=m.$
\end{enumerate}
\end{enumerate}
\end{theorem}

\begin{proof}
By \eqref{112},
\begin{align}
\lambda I_m + AA^{-} - BB^{-}  = \lambda I_m + AA^{\dag} - BB^{\dag} + AV_1E_{A} - BV_2E_{B},
 \label{27}
\end{align}
where $V_1\in {\mathbb C}^{n \times m}$ and $V_2\in {\mathbb C}^{p\times m}$ are arbitrary$.$
Applying \eqref{110} and \eqref{111} to \eqref{27} and simplifying gives
\begin{align*}
&  \max_{V_1, \, V_2}\!r(\, \lambda I_m + AA^{\dag} - BB^{\dag} + AV_1E_{A} - BV_2E_{B}, \,) \nb
 \\
& = \min\!\left\{ r[\, \lambda I_m + AA^{\dag} - BB^{\dag}, \, A, \, B \,], \
 r\!\begin{bmatrix} \lambda I_m + AA^{\dag} - BB^{\dag}  \\ E_{A} \\ E_{B} \end{bmatrix}\!,
\right. \nb
  \\
&  \ \ \ \ \ \  \ \ \ \  \left.  r\!\begin{bmatrix}  \lambda I_m + AA^{\dag} - BB^{\dag}
& A\\ E_{B}   & 0 \end{bmatrix}\!, \
r\!\begin{bmatrix} \lambda I_m + AA^{\dag} - BB^{\dag} & B \\ E_{A}  & 0 \end{bmatrix} \right\} \nb
\\
& = \min\!\left\{ r[\, \lambda I_m, \, A, \, B \,], \ \
 r\!\begin{bmatrix} \lambda I_m   \\ E_{A} \\ E_{B} \end{bmatrix}\!, \ \   r\!\begin{bmatrix}  (\lambda -1)I_m & A \\ E_{B} & 0 \end{bmatrix}\!, \ \
r\!\begin{bmatrix} (\lambda +1)I_m  & B \\ E_{A}  & 0 \end{bmatrix} \right\},
\end{align*}
and
\begin{align*}
& \min_{V_1, \, V_2}\!r(\, \lambda I_m + AA^{\dag} - BB^{\dag} + AV_1E_{A} - BV_2E_{B}, \,) \nb
\\
& = r\!\begin{bmatrix}  \lambda I_m + AA^{\dag} - BB^{\dag}  \\ E_{A} \\ E_{B} \end{bmatrix} + r[ \, \lambda I_m + AA^{\dag} - BB^{\dag}, \, A, \,  B \,] \nb
 \\
& \ \  + \ \max\left\{\! r\!\begin{bmatrix}\lambda I_m + AA^{\dag} - BB^{\dag} & A\\ E_{B} & 0\end{bmatrix} - r\!\begin{bmatrix} \lambda I_m + AA^{\dag} - BB^{\dag}
& A & B  \\ E_{B}   & 0 & 0 \end{bmatrix}
-r\!\begin{bmatrix} \lambda I_m + AA^{\dag} - BB^{\dag}
& A \\ E_{A}   & 0 \\ E_{B} & 0 \end{bmatrix},  \right. \nb
\\
& \ \ \ \ \ \ \ \ \ \ \ \ \ \ \ \ \ \ \ \  \left.  r\!\begin{bmatrix}  \lambda I_m + AA^{\dag} - BB^{\dag}
& B \\ E_{A} & 0 \end{bmatrix}  -r\!\begin{bmatrix} \lambda I_m + AA^{\dag} - BB^{\dag}
& A & B  \\ E_{A}   & 0 & 0 \end{bmatrix} - r\!\begin{bmatrix} \lambda I_m + AA^{\dag} - BB^{\dag}
& B \\ E_{A}   & 0 \\ E_{B} & 0  \end{bmatrix} \!\right\} \nb
\end{align*}
\begin{align*}
&  =  r\!\begin{bmatrix} \lambda I_m   \\ E_{A} \\ E_{B} \end{bmatrix} + r[\, \lambda I_m, \, A, \, B \,] + \ \max\left\{\! r\!\begin{bmatrix}  (\lambda -1)I_m & A \\ E_{B} & 0 \end{bmatrix} - r\!\begin{bmatrix} \lambda I_m & A & B  \\ E_{B}   & 0 & 0 \end{bmatrix}
-r\!\begin{bmatrix} (\lambda -1)I_m & A \\ E_{A}   & 0 \\ E_{B} & 0 \end{bmatrix},  \right. \nb
\\
& \ \ \ \ \ \ \ \ \ \ \ \ \ \ \  \ \ \ \ \ \ \ \ \ \ \ \ \ \ \ \ \ \ \ \ \ \ \ \ \ \ \ \ \ \ \ \ \ \  \left.
r\!\begin{bmatrix} (\lambda +1)I_m  & B \\ E_{A}  & 0 \end{bmatrix}  -r\!\begin{bmatrix} \lambda I_m
& A & B  \\ E_{A}   & 0 & 0 \end{bmatrix} - r\!\begin{bmatrix} (\lambda +1)I_m
& B \\ E_{A}   & 0 \\ E_{B} & 0  \end{bmatrix} \!\right\}.
\end{align*}
Substituting different values of $\lambda$ into the above formulas and simplifying yield the rank formulas required.
\end{proof}

\begin{theorem} \label{TN46}
Let $A \in {\mathbb C}^{m \times n}$  and $C \in {\mathbb C}^{p \times m}$ be given$.$
\begin{enumerate}
\item[{\rm (a)}] If $\lambda \neq 0, -1, -2,$  then the following two formulas hold
\begin{align*}
& \max_{A^{-},C^{-}}\!\!r(\, \lambda I_m + AA^{-} + C^{-}C\,) =  m,
\\
& \min_{A^{-},C^{-}}\!\!r(\,\lambda I_m + AA^{-} + C^{-}C\,) = \max\{ \, m - r(CA), \ \ r(A) + r(C) - r(CA) \,\}.
\end{align*}
In particular$,$ the following results hold$.$
\begin{enumerate}
\item[{\rm (i)}] There always exist $A^{-}$ and  $C^{-}$ such that $\lambda I_m + AA^{-} + C^{-}C$ is nonsingular$.$

\item[{\rm (ii)}]  $\lambda I_m + AA^{-} + C^{-}C$ is  nonsingular for all $A^{-}$ and $C^{-}$ $\Leftrightarrow$
the rank of $\lambda I_m + AA^{-} + C^{-}C$ is invariant for all $A^{-}$ and $C^{-}$ $\Leftrightarrow$
 $r[\, A, \, C \,] = r(A) = r(C)$ $\Leftrightarrow$ $CA =0$ and $r(A) + r(C) = m.$

\item[{\rm (iii)}] There do not exist $A^{-}$ and $C^{-}$ such that $\lambda I_m + AA^{-} + C^{-}C =0.$
\end{enumerate}

\item[{\rm (b)}] The following two formulas hold
\begin{align*}
& \max_{A^{-},C^{-}}\!\!r(\, AA^{-} + C^{-}C\,) =  \min\{ \, m, \ \ r(A) + r(C) \,\},
\\
& \min_{A^{-},C^{-}}\!\!r(\,AA^{-} + C^{-}C\,) = r(A) + r(C) - r(CA).
\end{align*}
In particular$,$ the following results hold$.$
\begin{enumerate}
\item[{\rm (i)}] There exist $A^{-}$ and  $C^{-}$ such that $AA^{-} + C^{-}C$ is nonsingular $\Leftrightarrow$
 $r(A) + r(C) \geq m.$

\item[{\rm (ii)}]  $AA^{-} + C^{-}C$ is  nonsingular for all $A^{-}$ and $C^{-}$ $\Leftrightarrow$
$r(CA) = r(A) + r(C) - m.$

\item[{\rm (iii)}] There exist $A^{-}$ and  $C^{-}$ such that $AA^{-} + C^{-}C =0$ $\Leftrightarrow$
$AA^{-} + C^{-}C =0$ for all $A^{-}$ and $C^{-}$ $\Leftrightarrow$  $A = 0$ and $C= 0.$

\item[{\rm (iv)}] The rank of $AA^{-} + C^{-}C$ is invariant for all $A^{-}$ and $C^{-}$ $\Leftrightarrow$ $CA =0$ or $r(CA) = r(A) + r(C) - m.$
\end{enumerate}

\item[{\rm (c)}] The following two formulas hold
\begin{align*}
& \max_{A^{-},C^{-}}\!\!r(\, -I_m + AA^{-} + C^{-}C\,) =  \min\{ \, m + r(CA) - r(A) , \ \  m + r(CA) - r(C) \,\},
\\
& \min_{A^{-},C^{-}}\!\!r(\,-I_m + AA^{-} + C^{-}C\,) = \max\{ \, m + r(CA) - r(A) - r(C), \ \ r(CA) \,\}.
\end{align*}
In particular$,$ the following results hold$.$
\begin{enumerate}
\item[{\rm (i)}] There exist $A^{-}$ and  $C^{-}$ such that $-I_m + AA^{-} + C^{-}C$ is nonsingular $\Leftrightarrow$
 $r(CA) = r(A) = r(C).$

\item[{\rm (ii)}]  $-I_m + AA^{-} + C^{-}C$ is  nonsingular for all $A^{-}$ and $C^{-}$ $\Leftrightarrow$
$A = 0$ and $C = 0,$  or $r(A) = r(C) = m.$

\item[{\rm (iii)}] There exist $A^{-}$ and  $C^{-}$ such that $AA^{-} + C^{-}C = I_m$ $\Leftrightarrow$  $CA =0$ and $r(A) + r(C) = m.$

\item[{\rm (iv)}]  $AA^{-} + C^{-}C = I_m$ cannot hold for all $A^{-}$ and $C^{-}.$

\item[{\rm (v)}] The rank of $-I_m + AA^{-} + C^{-}C$ is invariant for all $A^{-}$ and $C^{-}$ $\Leftrightarrow$ $A =0$ or $C =0$ or $r(A)= m$ or $r(C) =m.$
\end{enumerate}

\item[{\rm (d)}] The following two formulas hold
\begin{align*}
& \max_{A^{-},C^{-}}\!\!r(\, -2I_m + AA^{-} + C^{-}C\,) =  \min\{ \, m, \ \  2m - r(A) - r(C) \,\},
\\
& \min_{A^{-},C^{-}}\!\!r(\,-2I_m + AA^{-} + C^{-}C\,) = m - r(CA).
\end{align*}
In particular$,$ the following results hold$.$
\begin{enumerate}
\item[{\rm (i)}] There exist $A^{-}$ and  $C^{-}$ such that $-2I_m + AA^{-} + C^{-}C$ is nonsingular $\Leftrightarrow$
 $r(A) + r(C) \leq m.$

\item[{\rm (ii)}]  $-2I_m + AA^{-} + C^{-}C$ is  nonsingular for all $A^{-}$ and $C^{-}$ $\Leftrightarrow$
$CA = 0.$

\item[{\rm (iii)}] There exist $A^{-}$ and  $C^{-}$ such that $AA^{-} + C^{-}C = 2I_m$ $\Leftrightarrow$
the rank of $-I_m + AA^{-} + C^{-}C$ is invariant for all $A^{-}$ and $C^{-}$ $\Leftrightarrow$
 $r(A) = r(C) = m.$

\item[{\rm (iv)}] The rank of $-2I_m + AA^{-} + C^{-}C$ is invariant for all $A^{-}$ and $C^{-}$ $\Leftrightarrow$ $CA =0$ or $r(CA) = r(A) + r(C) - m.$
\end{enumerate}
\end{enumerate}
\end{theorem}

\begin{proof}
By \eqref{112},
\begin{align}
&\lambda I_m + AA^{-} + C^{-}C  = \lambda I_m + AA^{\dag} + C^{\dag}C + AV_1E_{A} + F_{C}V_2C,
 \label{38}
\end{align}
where $V_1\in {\mathbb C}^{n \times m}$ and $V_2\in {\mathbb C}^{m\times p}$ are arbitrary$.$
Applying \eqref{110} and \eqref{111} to \eqref{38} and simplifying gives
\begin{align*}
&  \max_{V_1, \, V_2}\!r(\, \lambda I_m + AA^{\dag} + C^{\dag}C + AV_1E_{A} + F_{C}V_2C \,) \nb
 \\
& = \min\!\left\{ r[\, \lambda I_m + AA^{\dag} + C^{\dag}C, \, A, \, F_{C} \,], \
 r\!\begin{bmatrix} \lambda I_m + AA^{\dag} + C^{\dag}C  \\ E_{A} \\ C \end{bmatrix}\!,  \right. \nb
\\
&  \ \ \ \ \ \  \ \ \ \left. r\!\begin{bmatrix}  \lambda I_m + AA^{\dag} + C^{\dag}C
& A\\ C   & 0 \end{bmatrix}\!, \
r\!\begin{bmatrix} \lambda I_m + AA^{\dag} + C^{\dag}C & F_{C} \\ E_{A}  & 0 \end{bmatrix} \right\} \nb
\\
& = \min\!\left\{ r[\, (\lambda +1)I_m, \, A, \, F_{C} \,], \ \
 r\!\begin{bmatrix} (\lambda + 1)I_m
 \\ E_{A} \\ C \end{bmatrix}, \ \  r\!\begin{bmatrix} \lambda I_m & A \\ C   & 0 \end{bmatrix}, \ \   r\!\begin{bmatrix} (\lambda +2)I_m
& F_{C} \\ E_{A}   & 0 \end{bmatrix} \right\}\!,
\end{align*}
and
\begin{align*}
& \min_{V_1, \, V_2}\!r(\, \lambda I_m + AA^{\dag} + C^{\dag}C + AV_1E_{A} + F_{C}V_2C \,) = r\!\begin{bmatrix}  \lambda I_m
+ AA^{\dag} + C^{\dag}C  \\ E_{A} \\ C \end{bmatrix} + r[ \, \lambda I_m + AA^{\dag} + C^{\dag}C, \, A, \,  F_{C} \,] \nb
 \\
& \ \  + \ \max\left\{\! r\!\begin{bmatrix}\lambda I_m + AA^{\dag} + C^{\dag}C & A\\ C & 0\end{bmatrix} - r\!\begin{bmatrix} \lambda I_m + AA^{\dag} + C^{\dag}C
& A & F_{C}  \\ C   & 0 & 0 \end{bmatrix}
-r\!\begin{bmatrix} \lambda I_m + AA^{\dag} + C^{\dag}C
& A \\ E_{A}   & 0 \\ C & 0 \end{bmatrix}\!,  \right. \nb
\\
& \ \ \ \ \ \ \ \ \ \ \ \ \ \ \  \left.  r\!\begin{bmatrix}  \lambda I_m + AA^{\dag} + C^{\dag}C
& F_{C} \\ E_{A} & 0 \end{bmatrix}  -r\!\begin{bmatrix} \lambda I_m + AA^{\dag} + C^{\dag}C
& A & F_{C}  \\ E_{A}   & 0 & 0 \end{bmatrix} - r\!\begin{bmatrix} \lambda I_m + AA^{\dag} + C^{\dag}C
& F_{C} \\ E_{A}   & 0 \\ C & 0  \end{bmatrix} \!\right\} \nb
\\
&  = r\!\begin{bmatrix} (\lambda + 1)I_m
 \\ E_{A} \\ C \end{bmatrix} + r[\, (\lambda +1)I_m, \, A, \, F_{C} \,]
  +  \max\left\{\! r\!\begin{bmatrix}\lambda I_m & A\\ C & 0\end{bmatrix} - r\!\begin{bmatrix} \lambda I_m
& A & F_{C}  \\ C   & 0 & 0 \end{bmatrix}
-r\!\begin{bmatrix} \lambda I_m
& A \\ E_{A}   & 0 \\ C & 0 \end{bmatrix}\!,  \right. \nb
\\
& \ \ \ \ \ \ \ \ \ \ \ \ \ \ \  \ \ \ \ \ \  \ \ \ \ \ \ \ \ \ \ \ \ \left.  r\!\begin{bmatrix}  (\lambda +2)I_m
& F_{C} \\ E_{A} & 0 \end{bmatrix}  -r\!\begin{bmatrix} (\lambda +1)I_m
& A & F_{C}  \\ E_{A}   & 0 & 0 \end{bmatrix} - r\!\begin{bmatrix} (\lambda +1)I_m
& F_{C} \\ E_{A}   & 0 \\ C & 0  \end{bmatrix} \!\right\}\!.
\end{align*}
Substituting different values of $\lambda$ into the above formulas and simplifying yield the rank formulas required.
\end{proof}

\begin{theorem} \label{T7}
Let $A \in {\mathbb C}^{m \times n}$  and $C \in {\mathbb C}^{p \times m}$ be given$.$
\begin{enumerate}
\item[{\rm (a)}] If $\lambda \neq 1, 0, -1,$  then the following two formulas hold
\begin{align*}
& \max_{A^{-},C^{-}}\!\!r(\, \lambda I_m + AA^{-} - C^{-}C\,) =  m,
\\
& \min_{A^{-},C^{-}}\!\!r(\,\lambda I_m + AA^{-} - C^{-}C\,) = \max\{ \, m - r(CA), \ \ r(A) + r(C) - r(CA) \,\}.
\end{align*}
In particular$,$ the following results hold$.$
\begin{enumerate}
\item[{\rm (i)}] There always exist $A^{-}$ and $C^{-}$ such that $\lambda I_m + AA^{-} - C^{-}C$ is nonsingular$.$

\item[{\rm (ii)}] $\lambda I_m + AA^{-} - C^{-}C$ is nonsingular for all $A^{-}$ and $C^{-}$ $\Leftrightarrow$
the rank of $\lambda I_m + AA^{-} - C^{-}C$ is invariant for all $A^{-}$ and $C^{-}$ $\Leftrightarrow$
 $CA =0$ and $r(A) + r(C) = m.$

\item[{\rm (iii)}] There do not exist $A^{-}$ and $C^{-}$ such that $\lambda I_m + AA^{-} - C^{-}C =0.$
\end{enumerate}

\item[{\rm (b)}] The following two formulas hold
\begin{align*}
& \max_{A^{-},C^{-}}\!\!r(\, AA^{-} - C^{-}C\,) =  m - |r(A) + r(C) - m|,
\\
& \min_{A^{-},C^{-}}\!\!r(\,AA^{-} - C^{-}C\,) = r(A) + r(C) - 2r(CA).
\end{align*}
In particular$,$ the following results hold$.$
\begin{enumerate}
\item[{\rm (i)}] There exist $A^{-}$ and  $C^{-}$ such that $AA^{-} - C^{-}C$ is nonsingular $\Leftrightarrow$
 $r(A) + r(C) =m.$

\item[{\rm (ii)}]  $AA^{-} - C^{-}C$ is  nonsingular for all $A^{-}$ and $C^{-}$ $\Leftrightarrow$
$CA = 0$ and $r(A) + r(C) =m.$

\item[{\rm (iii)}] There exist $A^{-}$ and  $C^{-}$ such that $AA^{-} = C^{-}C$ $\Leftrightarrow$  $r(CA) = r(A) = r(C).$

\item[{\rm (iv)}]   $AA^{-} = C^{-}C$ holds for all $A^{-}$ and $C^{-}$ $\Leftrightarrow$  $r(A) = r(C) =m.$

\item[{\rm (v)}] The rank of $AA^{-} - C^{-}C$ is invariant for all $A^{-}$ and $C^{-}$ $\Leftrightarrow$ $CA = 0$ or $r(A) + r(C) =m.$
\end{enumerate}

\item[{\rm (c)}] The following two formulas hold
\begin{align*}
& \max_{A^{-},C^{-}}\!\!r(\, -I_m + AA^{-} - C^{-}C\,) = m - r(A) + r(CA),
\\
& \min_{A^{-},C^{-}}\!\!r(\,-I_m + AA^{-} - C^{-}C\,) = \max\{ \, m - r(A), \ \  r(C) \,\}.
\end{align*}
In particular$,$ the following results hold$.$
\begin{enumerate}
\item[{\rm (i)}] There exist $A^{-}$ and  $C^{-}$ such that $-I_m + AA^{-} - C^{-}C$ is nonsingular $\Leftrightarrow$
 $r(A) = r(CA).$

\item[{\rm (ii)}]  $-I_m + AA^{-} - C^{-}C$ is  nonsingular for all $A^{-}$ and $C^{-}$ $\Leftrightarrow$
$A =0$  or $r(C) =m.$

\item[{\rm (iii)}] There exist $A^{-}$ and  $C^{-}$ such that $AA^{-} - C^{-}C =I_m$ $\Leftrightarrow$
 $AA^{-} - C^{-}C = I_m$ holds for all $A^{-}$ and $C^{-}$ $\Leftrightarrow$
$r(A) = m$ and $C =0.$

\item[{\rm (iv)}] The rank of $-I_m + AA^{-} - C^{-}C$ is invariant for all $A^{-}$ and $C^{-}$ $\Leftrightarrow$ $CA = 0$ or $r(CA) = r(A) + r(C) - m.$
\end{enumerate}
\end{enumerate}
\end{theorem}

\begin{proof}
By \eqref{112},
\begin{align}
&\lambda I_m + AA^{-} - C^{-}C  = \lambda I_m + AA^{\dag} - C^{\dag}C + AV_1E_{A} - F_{C}V_2C,
 \label{47}
\end{align}
where $V_1\in {\mathbb C}^{n \times m}$ and $V_2\in {\mathbb C}^{m\times p}$ are arbitrary$.$
Applying \eqref{110} and \eqref{111} to \eqref{47} and simplifying gives
\begin{align*}
&  \max_{V_1, \, V_2}\!r(\, \lambda I_m + AA^{\dag} - C^{\dag}C + AV_1E_{A} - F_{C}V_2C \,) \nb
 \\
& = \min\!\left\{ r[\, \lambda I_m + AA^{\dag} - C^{\dag}C, \, A, \, F_{C} \,], \
 r\!\begin{bmatrix} \lambda I_m + AA^{\dag} - C^{\dag}C  \\ E_{A} \\ C \end{bmatrix}\!, \right. \nb
\\
& \ \ \ \  \ \ \ \ \ \  \left.  r\!\begin{bmatrix}  \lambda I_m + AA^{\dag} - C^{\dag}C
& A \\ C   & 0 \end{bmatrix}\!, \
r\!\begin{bmatrix} \lambda I_m + AA^{\dag} - C^{\dag}C & F_{C} \\ E_{A}  & 0 \end{bmatrix} \right\} \nb
\\
& = \min\!\left\{ r[\, (\lambda -1)I_m, \, A, \, F_{C} \,], \
 r\!\begin{bmatrix} (\lambda +1)I_m  \\ E_{A} \\ C \end{bmatrix}\!,  \ \  r\!\begin{bmatrix}  \lambda I_m
& A \\ C   & 0 \end{bmatrix}\!, \
r\!\begin{bmatrix} \lambda I_m & F_{C} \\ E_{A}  & 0 \end{bmatrix} \right\}\!,
\end{align*}
and
\begin{align*}
& \min_{V_1, \, V_2}\!r(\, \lambda I_m + AA^{\dag} - C^{\dag}C + AV_1E_{A} - F_{C}V_2C \,) = r\!\begin{bmatrix}  \lambda I_m + AA^{\dag} - C^{\dag}C  \\ E_{A} \\ C \end{bmatrix} + r[ \, \lambda I_m + AA^{\dag} - C^{\dag}C, \, A, \,  F_{C} \,] \nb
 \\
& \ \  + \ \max\left\{\! r\!\begin{bmatrix}\lambda I_m + AA^{\dag} - C^{\dag}C & A\\ C & 0\end{bmatrix} - r\!\begin{bmatrix} \lambda I_m + AA^{\dag} - C^{\dag}C
& A & F_{C}  \\ C   & 0 & 0 \end{bmatrix}
-r\!\begin{bmatrix} \lambda I_m + AA^{\dag} - C^{\dag}C
& A \\ E_{A}   & 0 \\ C & 0 \end{bmatrix}\!,  \right. \nb
\\
& \ \ \ \ \ \ \ \ \ \ \ \ \ \ \  \left.  r\!\begin{bmatrix}  \lambda I_m + AA^{\dag} - C^{\dag}C
& F_{C} \\ E_{A} & 0 \end{bmatrix}  -r\!\begin{bmatrix} \lambda I_m + AA^{\dag} - C^{\dag}C
& A & F_{C}  \\ E_{A}   & 0 & 0 \end{bmatrix} - r\!\begin{bmatrix} \lambda I_m + AA^{\dag} - C^{\dag}C
& F_{C} \\ E_{A}   & 0 \\ C & 0  \end{bmatrix} \!\right\} \nb
\\
& = r\!\begin{bmatrix}  (\lambda +1)I_m  \\ E_{A} \\ C \end{bmatrix} + r[ \, (\lambda -1)I_m, \, A, \,  F_{C} \,] \nb
 \\
& \ \  + \ \max\left\{\! r\!\begin{bmatrix}\lambda I_m  & A\\ C & 0\end{bmatrix} - r\!\begin{bmatrix} \lambda I_m
& A & F_{C}  \\ C   & 0 & 0 \end{bmatrix}
-r\!\begin{bmatrix} \lambda I_m & A \\ E_{A}   & 0 \\ C & 0 \end{bmatrix}, \ r\!\begin{bmatrix}  \lambda I_m
& F_{C} \\ E_{A} & 0 \end{bmatrix}  -r\!\begin{bmatrix} \lambda I_m
& A & F_{C}  \\ E_{A}   & 0 & 0 \end{bmatrix} - r\!\begin{bmatrix} \lambda I_m
& F_{C} \\ E_{A}   & 0 \\ C & 0  \end{bmatrix} \!\right\}\!.
\end{align*}
Substituting different values of $\lambda$ into the above formulas and simplifying yield the rank formulas required.
\end{proof}

Setting $C= A$ in Theorems \ref{TK35} and \ref{TK36} leads to the following results.

\begin{corollary} \label{T8}
Let $A \in {\mathbb C}^{m \times m}$ be given$.$
\begin{enumerate}
\item[{\rm (a)}] If $\lambda \neq 0, -1, -2,$  then the following two formulas hold
\begin{align*}
& \max_{A^{-}}r(\, \lambda I_m + AA^{-} + A^{-}A\,) =  m,
\\
& \min_{A^{-}}r(\,\lambda I_m + AA^{-} + A^{-}A\,) = \max\{ \, m - r(A^2), \ \ 2r(A) - r(A^2) \,\}.
\end{align*}
In particular$,$ the following results hold$.$
\begin{enumerate}
\item[{\rm (i)}] There always exists $A^{-}$ such that $\lambda I_m + AA^{-} + A^{-}A$ is nonsingular$.$

\item[{\rm (ii)}]  $\lambda I_m + AA^{-} + A^{-}A$ is  nonsingular for all $A^{-}$ $\Leftrightarrow$
the rank of $\lambda I_m + AA^{-} + A^{-}A$ is invariant for all $A^{-}$ $\Leftrightarrow$ $A^2 =0$ and $2r(A) = m.$

\item[{\rm (iii)}] There does not exist $A^{-}$ such that $\lambda I_m + AA^{-} + A^{-}A =0.$
\end{enumerate}

\item[{\rm (b)}] The following two formulas hold
\begin{align*}
& \max_{A^{-}}r(\, AA^{-} + A^{-}A\,) =  \min\{ \, m, \ \ 2r(A) \,\},
\\
& \min_{A^{-}}r(\,AA^{-} + A^{-}A\,) = 2r(A) - r(A^2).
\end{align*}
In particular$,$ the following results hold$.$
\begin{enumerate}
\item[{\rm (i)}] There exists $A^{-}$ such that $AA^{-} + A^{-}A$ is nonsingular $\Leftrightarrow$
 $2r(A) \geq m.$

\item[{\rm (ii)}]  $AA^{-} + A^{-}A$ is  nonsingular for all $A^{-}$ $\Leftrightarrow$
$r(A^2) = 2r(A) - m.$

\item[{\rm (iii)}] There exists $A^{-}$ such that $AA^{-} + A^{-}A =0$ $\Leftrightarrow$
$AA^{-} + A^{-}A =0$ for all $A^{-}$ $\Leftrightarrow$  $A = 0.$

\item[{\rm (iv)}] The rank of $AA^{-} + A^{-}A$ is invariant for all $A^{-}$ $\Leftrightarrow$ $A^2 =0$ or $r(A^2) = 2r(A) - m.$
\end{enumerate}

\item[{\rm (c)}] The following two formulas hold
\begin{align*}
& \max_{A^{-}}r(\, -I_m + AA^{-} + A^{-}A\,) =  \min\{ \, m + r(A^2) - r(A) , \ \  m + r(A^2) - r(A) \,\},
\\
& \min_{A^{-}}r(\,-I_m + AA^{-} + A^{-}A\,) = \max\{ \, m + r(A^2) - 2r(A), \ \ r(A^2) \,\}.
\end{align*}
In particular$,$ the following results hold$.$
\begin{enumerate}
\item[{\rm (i)}] There exists $A^{-}$ such that $-I_m + AA^{-} + A^{-}A$ is nonsingular $\Leftrightarrow$ $r(A^2) = r(A).$

\item[{\rm (ii)}]  $-I_m + AA^{-} + A^{-}A$ is  nonsingular for all $A^{-}$ $\Leftrightarrow$
$A = 0$ or $r(A) = m.$

\item[{\rm (iii)}] There exists $A^{-}$ such that $AA^{-} + A^{-}A = I_m$ $\Leftrightarrow$
$A^2 =0$ and $2r(A) = m.$

\item[{\rm (iv)}]  $AA^{-} + A^{-}A = I_m$ cannot hold for all $A^{-}.$

\item[{\rm (v)}] The rank of $-I_m + AA^{-} + A^{-}A$ is invariant for all $A^{-}$ $\Leftrightarrow$ $A =0$ or $r(A) =m.$
\end{enumerate}

\item[{\rm (d)}] The following two formulas hold
\begin{align*}
& \max_{A^{-}}r(\, -2I_m + AA^{-} + A^{-}A\,) =  \min\{ \, m, \ \  2m - 2r(A) \,\},
\\
& \min_{A^{-}}r(\,-2I_m + AA^{-} + A^{-}A\,) = m - r(A^2).
\end{align*}
In particular$,$ the following results hold$.$
\begin{enumerate}
\item[{\rm (i)}] There exists $A^{-}$ such that $-2I_m + AA^{-} + A^{-}A$ is nonsingular $\Leftrightarrow$
 $2r(A) \leq m.$

\item[{\rm (ii)}]  $-2I_m + AA^{-} + A^{-}A$ is  nonsingular for all $A^{-}$ $\Leftrightarrow$
$A^2 = 0.$

\item[{\rm (iii)}] There exists $A^{-}$ such that $AA^{-} + A^{-}A = 2I_m$ $\Leftrightarrow$
the rank of $-I_m + AA^{-} + A^{-}A$ is invariant for all $A^{-}$ $\Leftrightarrow$
 $r(A) = m.$

\item[{\rm (iv)}] The rank of $-2I_m + AA^{-} + A^{-}A$ is invariant for all $A^{-}$ $\Leftrightarrow$ $A^2 =0$ or $r(A^2) = 2r(A) - m.$
\end{enumerate}
\end{enumerate}
\end{corollary}

\begin{corollary} \label{T9}
Let $A \in {\mathbb C}^{m \times m}$ be given$.$
\begin{enumerate}
\item[{\rm (a)}] If $\lambda \neq 1, 0, -1,$  then the following two formulas hold
\begin{align*}
& \max_{A^{-}}r(\, \lambda I_m + AA^{-} - A^{-}A\,) =  m,
\\
& \min_{A^{-}}r(\,\lambda I_m + AA^{-} - A^{-}A\,) = \max\{ \, m - r(A^2), \ \ 2r(A) - r(A^2) \,\}.
\end{align*}
In particular$,$ the following results hold$.$
\begin{enumerate}
\item[{\rm (i)}] There always exists $A^{-}$ such that $\lambda I_m + AA^{-} - A^{-}A$ is nonsingular$.$

\item[{\rm (ii)}] $\lambda I_m + AA^{-} - A^{-}A$ is nonsingular for all $A^{-}$ $\Leftrightarrow$
the rank of $\lambda I_m + AA^{-} - A^{-}A$ is invariant for all $A^{-}$ $\Leftrightarrow$ $A^2 =0$ and $2r(A) = m.$

\item[{\rm (iii)}] There do not exist $A^{-}$ such that $\lambda I_m + AA^{-} - A^{-}A =0.$
\end{enumerate}

\item[{\rm (b)}] The following two formulas hold
\begin{align*}
& \max_{A^{-}}r(\, AA^{-} - A^{-}A\,) =  m - |2r(A) - m|,
\\
& \min_{A^{-}}r(\,AA^{-} - A^{-}A\,) = 2r(A) - 2r(A^2).
\end{align*}
In particular$,$ the following results hold$.$
\begin{enumerate}
\item[{\rm (i)}] There exists $A^{-}$ such that $AA^{-} - A^{-}A$ is nonsingular $\Leftrightarrow$
 $2r(A) =m.$

\item[{\rm (ii)}]  $AA^{-} - A^{-}A$ is  nonsingular for all $A^{-}$ $\Leftrightarrow$
$A^2 = 0$ and $2r(A) =m.$

\item[{\rm (iii)}] There exists $A^{-}$ such that $AA^{-} = A^{-}A$ $\Leftrightarrow$  $r(A^2) = r(A).$

\item[{\rm (iv)}]   $AA^{-} = A^{-}A$ holds for all $A^{-}$ $\Leftrightarrow$  $r(A) =m.$

\item[{\rm (v)}] The rank of $AA^{-} - A^{-}A$ is invariant for all $A^{-}$ $\Leftrightarrow$ $A = 0$ or $2r(A) =m.$
\end{enumerate}

\item[{\rm (c)}] The following two formulas hold
\begin{align*}
& \max_{A^{-}}r(\, -I_m + AA^{-} - A^{-}A\,) = m - r(A) + r(A^2),
\\
& \min_{A^{-}}r(\,-I_m + AA^{-} - A^{-}A\,) = \max\{ \, m - r(A), \ \  r(A) \,\}.
\end{align*}
In particular$,$ the following results hold$.$
\begin{enumerate}
\item[{\rm (i)}] There exists $A^{-}$ such that $-I_m + AA^{-} - A^{-}A$ is nonsingular $\Leftrightarrow$
 $r(A^2) = r(A).$

\item[{\rm (ii)}]  $-I_m + AA^{-} - A^{-}A$ is  nonsingular for all $A^{-}$ $\Leftrightarrow$
$A =0$  or $r(A) =m.$

\item[{\rm (iii)}] There exists $A^{-}$ such that $AA^{-} - A^{-}A =I_m$ $\Leftrightarrow$
 $AA^{-} - A^{-}A = I_m$ holds for all $A^{-}$ $\Leftrightarrow$
 $A =0$ and $r(A) = m.$

\item[{\rm (iv)}] The rank of $-I_m + AA^{-} - A^{-}A$ is invariant for all $A^{-}$ $\Leftrightarrow$ $A^2 = 0$ or $r(A^2) = 2r(A) - m.$
\end{enumerate}
\end{enumerate}
\end{corollary}

\begin{corollary} \label{T10}
Let $A \in \mathbb C^{m \times n}, \, B \in \mathbb C^{m \times k},$ $C \in \mathbb C^{l \times n},$ and
$D \in \mathbb C^{l \times n}$  be given$,$ and denote $M = \begin{bmatrix} A & B \\ C & D  \end{bmatrix}$  and
$N =\begin{bmatrix} A & 0 \\ 0 & D  \end{bmatrix}.$  If $\lambda \neq 1, 0, -1,$  then the following two formulas hold
\begin{align}
& \max_{M^{-},N^{-}}\!\!r\!\left(\, \lambda I_{m+l} + MM^{-} - NN^{-} \right) =  m + l,
\label{361}
\\
& \min_{M^{-},N^{-}}\!\!r\!\left(\, \lambda I_{m+l} + MM^{-} - NN^{-} \right)  =  m + l - r[\, A, \, B \,] - r[\, C, \, D \,] + \max\left\{r(M), \ \   r(A) + r(D)  \right\}.
\label{362}
\end{align}
In particular$,$  the following formulas hold
\begin{align}
& \max_{M^{-},N^{-}}\!\!r(\, I_{m+l} + MM^{-} - NN^{-}\,) =\min \{\, m +l, \ \  m + l + r(M) - r(A) - r(D) \},
\label{363}
\\
& \min_{M^{-},N^{-}}\!\!r(\, I_{m+l} + MM^{-} - NN^{-}\,) =  m + l + r(M) - r[\, A, \, B \,] -  r[\, C, \, D \,],
\label{364}
\\
&\max_{M^{-},N^{-}}\!\!r(\, MM^{-} - NN^{-}\,) =   r[\, A, \, B \,] + r[\, C, \, D \,] + \min\{ \, 0, \ \  m + l  -  r(M) - r(A) - r(D) \},
\label{365}
\\
& \min_{M^{-},N^{-}}\!\!r(\, MM^{-} - NN^{-}\,) = r[\, A, \, B \,] + r[\, C, \, D \,] -  \min\{ \, r(M), \ \ r(A) + r(D) \}.
\label{366}
\end{align}
\end{corollary}

Setting the both sides of \eqref{362}--\eqref{366} equal to  $m + l$ or zero will lead to a group of facts  on the singularity, non-singularity, and equality of the corresponding matrix expressions. In addition, it is no doubt that the maximum and minimum ranks of the following matrix pencils
$$
 \lambda_1 I_m + \lambda_2 AA^{-} \pm  \lambda_3 BB^{-}, \ \  \lambda_1 I_m + \lambda_2  AA^{-}\pm \lambda_3 C^{-}C
$$
with respect to the choices of $A^{-}$, $B^{-}$, and $C^{-}$ can be determined by a similar approach.

\section[5]{Conclusions}

We have established some general matrix rank equalities using the block matrix method, and presented many applications of these formulas in dealing with idempotent matrices and their operations. These findings seem exciting and can be selected as constructive issues in textbooks and handbooks on matrices and linear algebra. This kind of work also shows that there exist still many simple but valuable problems on fundamental objects in linear algebra for which we can make deeper exploration and find out various novel and intrinsic conclusions.

Idempotents are really simple and funny in algebras and seem to appear everywhere, so that they are prominent issues for consideration in the field of mathematics. It is expected that numerous analytical formulas can be established for calculating ranks of matrix expressions composed by idempotent matrices using various tricky constructions of block matrices, which we believe will greatly enrich the classic issue in linear algebra from the bottom level, and will also provide applicable tools to deal with various challenging problems in matrix analysis and applications. In addition to this miscellaneous work, it would be interest to consider performances of the sums and differences of several idempotent matrices,  such as, $AA^{-} \pm  BB^{-} \pm  CC^{-}, \ \ [\, A, \, B \,][\, A, \, B \,]^{-} - AA^{-} -  BB^{-},
$, etc. In these situations, analytical formulas for calculating the maximum and minimum ranks of $A - B_1X_1C_1 -B_2X_2C_2 - B_3X_3C_3$ will be used. However, it was noticed by the present author in \cite{Tian:2000} that there may exist no analytical formulas for calculating the minimum ranks of general LMVFs with three or more variable matrices. Thus the rank distributions of these matrix pencils will remain open before solving the minimum rank problems on general LMVFs.

It is well known that idempotents and generalized inverses of elements can be defined in general algebraic structures, which are also important tools in the approaches of the algebraic structures; see e.g.,
\cite{AFPR,AL,AJL,AS,And,AAK,Andr1,Andr2,ACo,ACS,AC,ACG,ACMa,AMMZ,ASS,ASh,BSa,BBo,BES1,BES2,BES3,BES4,BES5,
BSe,Berg,Bha1,Bha2,BV,Bik1,Bik2,BHKP,BGK,BSS,BS1,BS2,BS3,BH,BBa,Bro,Buc1,Buc2,BVa,Cao,CX,CP,CDu,
CGM,CMa1,CMa2,CG,CGi,CM,CMS1,CMS2,CD,Daw,Dol,DRR,Eck,Ehr,ERSS,Emm1, Emm2,FD,Fil,Fit,For,GW,HSa,HO1,HO2,HK,Hen,HT,How1,How2,HL,IMO,ITP,Lja,Kal,Kaw,KH1,KH2,Kol,KCD,
KR1,KR2,KR3,KR4,Kir,KS,KRSa,Krup,KS1,KS2,KRSi,Laf,LP,LMR,Lin,LW,MOS,MV,Mus,Neu,Oml,PRW,Park,PT,Ped,Pet,PTu,
Pru,Pta,QFG,Rab0,Rab1,Rab2,RR,RSa,RSS,Ran,Rao,RS,RSh,RSi,Rui,SZ,ST1,
ST2,Sch,Se,Sel,Shc,SG,Spi,Str,TV,Vid1,Vid2,WW,WZ,WDD,Wei,Wu,XZ1,XZ2,Yos,ZW,Zel,Zem,Zuo1,Zuo2,Zuo3}. Thus it is believed that the preceding results and facts
 can be extended with some efforts to idempotents and generalized inverses of elements in other algebraic structures that are somehow close to the matrix case.

\end{document}